\pdfoutput=1
\documentclass[12pt]{amsart}

\usepackage{lmodern}
\usepackage{ifthen}
\usepackage{amsfonts}
\usepackage{amsxtra}
\usepackage{amssymb}
\usepackage{mathdots}
\usepackage{array}
\usepackage[margin=1in]{geometry}
\usepackage{xcolor}
\definecolor{cite}{rgb}{0.30,0.60,1.00}
\definecolor{url}{rgb}{0.00,0.00,0.80}
\definecolor{link}{rgb}{0.40,0.10,0.20}
\usepackage[colorlinks,linkcolor=link,urlcolor=url,citecolor=cite,pagebackref,breaklinks]{hyperref}
\usepackage{bbm}
\usepackage{mathtools}
\usepackage{mathrsfs}
\usepackage{appendix}
\usepackage[all]{xy}
\usepackage[lite,abbrev,msc-links,alphabetic]{amsrefs}

\usepackage{graphicx}
\usepackage{multirow}
\usepackage{pstricks}
\usepackage{pst-pdf}
\usepackage{enumitem}
\usepackage{pifont}
\usepackage{marvosym}
\usepackage{txfonts}

\usepackage[OT2,T1]{fontenc}
\DeclareSymbolFont{cyrletters}{OT2}{wncyr}{m}{n}
\DeclareMathSymbol{\Sha}{\mathalpha}{cyrletters}{"58}

\numberwithin{equation}{section}

\theoremstyle{plain}
\newtheorem{proposition}{Proposition}[section]
\newtheorem{conjecture}[proposition]{Conjecture}
\newtheorem{corollary}[proposition]{Corollary}
\newtheorem{lem}[proposition]{Lemma}
\newtheorem{theorem}[proposition]{Theorem}

\theoremstyle{definition}
\newtheorem{definition}[proposition]{Definition}

\newtheorem{notation}[proposition]{Notation}

\newtheorem{hypothesis}[proposition]{Hypothesis}

\theoremstyle{remark}
\newtheorem{remark}[proposition]{Remark}
\newtheorem{example}[proposition]{Example}

\renewcommand{\b}[1]{\mathbf{#1}}
\renewcommand{\c}[1]{\mathcal{#1}}
\renewcommand{\d}[1]{\mathbb{#1}}
\newcommand{\f}[1]{\mathfrak{#1}}
\renewcommand{\r}[1]{\mathrm{#1}}
\newcommand{\s}[1]{\mathscr{#1}}

\renewcommand{\(}{\left(}
\renewcommand{\)}{\right)}
\newcommand{\res}{\mathbin{|}}
\newcommand{\ol}[1]{\overline{#1}{}}
\newcommand{\wt}[1]{\widetilde{#1}{}}

\renewcommand{\leq}{\leqslant}
\renewcommand{\geq}{\geqslant}

\newcommand{\bZ}{\b Z}

\newcommand{\bh}{\b h}

\newcommand{\cE}{\c E}
\newcommand{\cF}{\c F}

\newcommand{\cI}{\c I}

\newcommand{\cS}{\c S}

\newcommand{\cV}{\c V}
\newcommand{\cW}{\c W}

\newcommand{\dA}{\d A}
\newcommand{\dB}{\d B}
\newcommand{\dC}{\d C}

\newcommand{\dL}{\d L}

\newcommand{\dN}{\d N}

\newcommand{\dQ}{\d Q}
\newcommand{\dR}{\d R}

\newcommand{\dV}{\d V}
\newcommand{\dW}{\d W}

\newcommand{\dZ}{\d Z}

\newcommand{\fS}{\f S}

\newcommand{\fU}{\f U}

\newcommand{\fc}{\f c}

\newcommand{\fp}{\f p}

\newcommand{\rH}{\r H}

\newcommand{\rU}{\r U}

\newcommand{\rd}{\r d}

\newcommand{\rh}{\r h}

\newcommand{\rt}{\r t}

\newcommand{\sK}{\s K}
\newcommand{\sL}{\s L}
\newcommand{\sM}{\s M}

\newcommand{\sP}{\s P}

\newcommand{\sS}{\s S}
\newcommand{\sT}{\s T}

\newcommand{\sX}{\s X}

\newcommand{\tP}{\mathtt{P}}

\newcommand{\tS}{\mathtt{S}}

\newcommand{\tV}{\mathtt{V}}

\newcommand{\tc}{\mathtt{c}}

\newcommand{\te}{\mathtt{e}}

\renewcommand{\tt}{\mathtt{t}}

\newcommand{\bmu}{\boldsymbol{\mu}}

\newcommand{\llangle}{\langle\!\langle}
\newcommand{\rrangle}{\rangle\!\rangle}

\newcommand{\pres}[2]{\prescript{#1}{}{#2}}

\newcommand{\ab}{\r{ab}}

\newcommand{\cris}{\r{cris}}

\newcommand{\dr}{\r{dR}}

\newcommand{\fin}{\r{fin}}

\newcommand{\inert}{\r{inert}}

\newcommand{\ram}{\r{ram}}

\newcommand{\spl}{\r{spl}}

\newcommand{\tor}{\r{tor}}

\newcommand{\unr}{\r{unr}}

\DeclareMathOperator{\As}{As}
\DeclareMathOperator{\Aut}{Aut}

\DeclareMathOperator{\CH}{CH}
\DeclareMathOperator{\Char}{char}

\DeclareMathOperator{\Cor}{Cores}

\DeclareMathOperator{\diag}{diag}

\DeclareMathOperator{\Gal}{Gal}

\DeclareMathOperator{\GL}{GL}

\DeclareMathOperator{\Hom}{Hom}

\DeclareMathOperator{\Ker}{ker}

\DeclareMathOperator{\Nm}{Nm}

\DeclareMathOperator{\RE}{Re}
\DeclareMathOperator{\Res}{Res}

\DeclareMathOperator{\Spec}{Spec}

\DeclareMathOperator{\Tr}{Tr}

\DeclareMathOperator{\vol}{vol}

\begin{document}

\title{Anticyclotomic $p$-adic $L$-functions for Rankin--Selberg product}

\author{Yifeng Liu}
\address{Institute for Advanced Study in Mathematics, Zhejiang University, Hangzhou 310058, China}
\email{liuyf0719@zju.edu.cn}

\date{\today}

\begin{abstract}
  We construct $p$-adic $L$-functions for Rankin--Selberg products of automorphic forms of hermitian type in the anticyclotomic direction for both root numbers. When the root number is $+1$, the construction relies on global Bessel periods on definite unitary groups which, due to the recent advances on the global Gan--Gross--Prasad conjecture, interpolate classical central $L$-values. When the root number is $-1$, we construct an element in the Iwasawa Selmer group using the diagonal cycle on the product of unitary Shimura varieties, and conjecture that its $p$-adic height interpolates derivatives of cyclotomic $p$-adic $L$-functions. We also propose the nonvanishing conjecture and the main conjecture in both cases.
\end{abstract}

\maketitle

\setcounter{tocdepth}{1}
\tableofcontents

\section{Introduction}
\label{ss:1}

In their pioneer work \cite{BD96}, Bertolini and Darmon constructed what they called Heegner distribution in both definite and indefinite cases for a classical weight two modular form and an imaginary quadratic field, opening the study of $p$-adic $L$-functions in the anticyclotomic direction. In the definite case, the product of Heegner distributions leads to a $p$-adic measure that interpolates central values of (classical) Rankin--Selberg $L$-functions, using the Waldspurger formula. In the indefinite case, the height of Heegner distribution leads to a $p$-adic measure that interpolates central derivatives of Rankin--Selberg $L$-functions, using the Gross--Zagier formula and its generalization. Since then, there have been numerous works addressing related problems including, for example, \cite{BD05} in the definite case and \cite{How04} in the indefinite case.

In this note, we make the first step to generalize these ideas, still along the anticyclotomic direction, to Rankin--Selberg products of higher ranks/dimensions, that is, beyond the $\GL_1\times\GL_2$-case. We consider a CM extension $E/F$, and a suitable admissible representation $\Pi=\Pi_n\boxtimes\Pi_{n+1}$ of $\GL_n(\dA_E^\infty)\times\GL_{n+1}(\dA_E^\infty)$ with coefficients in a $p$-adic field $\dL$ (see Section \ref{ss:ggp} for the precise setting; in particular, after base change to $\dC$, $\Pi$ is the finite part of a cohomological isobaric automorphic representation of minimal weights and with certain conjugate self-duality). Let $\tP$ be a set of $p$-adic places $v$ of $F$ split in $E$ such that $\Pi_v$ is ordinary (Definition \ref{de:ordinary}). Denote by $\cE_\tP/E$ the maximal anticyclotomic (with respect to $E/F$) extension of $E$ unramified outside (places above) $\tP$ with torsion free Galois group; and put $\Gamma_\tP\coloneqq\Gal(\cE_\tP/E)$, which is naturally a quotient of $E^\times\backslash(\dA_E^\infty)^\times$ and also a finite free $\dZ_p$-module. Possibly replacing $\dL$ by a finite extension, there is a Galois representation $\dW_\Pi$ of $E$ with coefficients in $\dL$ of rank $n(n+1)$ and weight $-1$ associated with $\Pi$ (see Example \ref{ex:galois}).

We may attach to $\Pi$ a ``root number'' $\epsilon(\Pi)\in\{\pm 1\}$ and call it coherent/incoherent if $\epsilon(\Pi)=+1$/$-1$, parallel to the definite/indefinite cases in \cite{BD96}.

When $\Pi$ is coherent, we show in Theorem \ref{th:function} that there exists a unique (bounded) measure $\sL^0_{\cE_\tP}(\Pi)\in\dL[[\Gamma_\tP]]^\circ\coloneqq\dZ_p[[\Gamma_\tP]]\otimes_{\dZ_p}\dL$ such that for every finite character $\chi\colon\Gamma_\tP\to\ol\dL^\times$ that is ramified at \emph{every} place above $\tP$ and every embedding $\iota\colon\ol\dL\to\dC$, we have
\[
\iota\sL^0_{\cE_\tP}(\Pi)(\chi)=\iota\omega\cdot
\frac{\Delta_{n+1}\cdot L(\tfrac{1}{2},\iota(\Pi_n\otimes\wt\chi)\times\iota\Pi_{n+1})}
{2^{d(\Pi_n)+d(\Pi_{n+1})}\cdot L(1,\iota\Pi_n,\As^{(-1)^n})L(1,\iota\Pi_{n+1},\As^{(-1)^{n+1}})}.
\]
Here, $\omega\in\dL^\times$ is an elementary factor depending on $\Pi_v$ and the conductor of $\chi_v$ for $v\in\tP$; $\Delta_{n+1}$ is a positive real constant in Definition \ref{de:relevant}(1); $\wt\chi$ is just $\chi$ but regarded as a character of $E^\times\backslash(\dA_E^\infty)^\times$; $d(\Pi_n)$ and $d(\Pi_{n+1})$ are certain positive integers introduced in Remark \ref{re:relevant}(4), same as the ones appearing in the refined Gan--Gross--Prasad conjecture \cite{Har14}. The proof relies on the recent advances on the global Gan--Gross--Prasad conjecture \cites{BPLZZ,BPCZ} and a local Birch lemma \cites{KMS00,Jan11}. We then propose a conjecture predicting when $\sL^0_{\cE_\tP}(\Pi)$ is nonvanishing (Conjecture \ref{co:vanishing1}), and an Iwasawa main conjecture for the Iwasawa Selmer group $\sX(\cE_\tP,\dW_\Pi)$ (Conjecture \ref{co:main1}).

When $\Pi$ is incoherent, we construct a collection of elements $(\kappa(\varphi))_\varphi$ in the compact Iwasawa Selmer group $\sS(\cE_\tP,\dW_\Pi)$ (Definition \ref{no:selmer}) of $\dW_\Pi$ along $\Gamma_\tP$, parameterized by ordinary vectors $\varphi\in\pi$. Here, $\pi$ is a certain ``descent'' of $\Pi$ to a unitary product group $G$ determined by the local Gan--Gross--Prasad conjecture (now theorem). The element $\kappa(\varphi)$ is constructed by variants of special cycles on the Shimura varieties associated with $G$ given by the natural diagonal subgroup $H\subseteq G$, generalizing the analogous construction in \cites{Ber95,BD96} using Heegner points. We then define a $p$-adic measure $\sL^1_{\cE_\tP}(\Pi)$ to be the Iwasawa $p$-adic height pairing between $\kappa(\varphi)$ and $\kappa(\varphi^\vee)$, modified by some local terms so that the result is independent of the choices of test vectors. Finally, we propose three conjectures: the nonvanishing conjecture for $\sL^1_{\cE_\tP}(\Pi)$ (Conjecture \ref{co:vanishing2}), an Iwasawa main conjecture (Conjecture \ref{co:main2}) and its variant (Conjecture \ref{co:main3}), and the interpolation conjecture to the derivative of cyclotomic $p$-adic $L$-function (Conjecture \ref{co:derivative}).

In our subsequent work \cite{LTX}, we show in many cases one side of the divisibility for the main conjectures for both root numbers (Conjectures \ref{co:main1} and \ref{co:main2}), generalizing \cite{BD05} and \cite{How04} to higher ranks/dimensions.

\subsection*{Notation and conventions}

\begin{itemize}
  \item Put $\dN\coloneqq\{0,1,2,\dots\}$.

  \item Denote by $\te$ a vector of length $1$.

  \item In this article, $\dL$ always denotes a field embeddable into $\dC$; and $\dL_?$ stands for an extension of $\dL$ that is again embeddable into $\dC$. We denote by $\ol\dQ$ the algebraic closure of $\dQ$ in $\dC$.
\end{itemize}

We fix a CM extension $E/F$ contained in $\dC$. Denote by
\begin{itemize}
  \item $\tc\in\Aut(\dC/F)$ the complex conjugation (which generates $\Gal(E/F)$),

  \item $E^\ab$ the maximal abelian extension of $E$ in $\dC$,

  \item $\eta_{E/F}\colon F^\times\backslash\dA_F^\times\to\{\pm1\}$ the associated quadratic character,

  \item $E^{\times-}$ the kernel of the norm map $\Nm_{E/F}\colon E^\times\to F^\times$,

  \item $\tV_F^{(w)}$ the set of places of $F$ above a place $w$ of $\dQ$,

  \item $\tV_F^\fin$ the set of non-archimedean places of $F$,

  \item $\tV_F^\spl$, $\tV_F^\inert$ and $\tV_F^\ram$ the subsets of $\tV_F^\fin$ of those that are split, inert and ramified in $E$, respectively,

  \item $\tV_E^?$ the set of places of $E$ above $\tV_F^?$,

  \item $\fp_v$ the maximal ideal of $O_{F_v}$ and $q_v$ the cardinality of $O_{F_v}/\fp_v$ for every $v\in\tV_F^\fin$.
\end{itemize}

We fix a prime number $p$, and take a (possibly empty) subset $\tP\subseteq\tV_F^{(p)}\cap\tV_F^\spl$ (which will be specified at a certain point). For every finite extension $K/F$, we denote by $\Gamma_\tP^K$ the maximal torsion free quotient of the profinite completion of
\[
K^\times\backslash(\dA_K^\infty)^\times\left/\prod_{v\in\tV_F^\fin\setminus\tP}(O_K\otimes_{O_F}O_{F_v})^\times\right.,
\]
which is naturally a finite free $\dZ_p$-module.

We have a homomorphism $\Nm_{E/F}^-\colon\Gamma_\tP^E\to\Gamma_\tP^E$ sending $a$ to $a/a^\tc$, whose image is denoted by $\Gamma_\tP$. Denote by $\cE_\tP\subseteq E^\ab$ the extension of $E$ such that $\Gal(\cE_\tP/E)=\Gamma_\tP$ under the global class field theory. For every function $\chi$ on $\Gamma_\tP$, we put $\wt\chi\coloneqq\chi\circ\Nm_{E/F}^-$ as a function on $\Gamma_\tP^E$.

\begin{definition}\label{de:extension}
We say that a field extension $\cE/E$ is a \emph{$\tP$-extension} if $\cE\subseteq\cE_\tP$; $\Gal(\cE/E)$ is torsion free; and the set of primes of $E$ ramified in $\cE$ consists exactly those above $\tP$.
\end{definition}

\subsection*{Acknowledgements}

The author would like to thank Daniel Disegni and David Loeffler for helpful comments, and also the anonymous referee for careful reading and useful commments. The research of Y.~L. is supported by National Key R\&D Program of China No. 2022YFA1005300.

\section{Gan--Gross--Prasad conjecture}
\label{ss:ggp}

\begin{definition}\label{de:relevant}
Let $N$ be a positive integer.
\begin{enumerate}
  \item Put
      \[
      \Delta_{N,v}\coloneqq\prod_{i=1}^N L(i,\eta_{E/F,v}^i)\in\dQ_{>0}
      \]
      for every $v\in\tV_F^\fin$; and put
      \[
      \Delta_N\coloneqq\prod_{i=1}^NL(i,(\eta_{E/F}^\infty)^i)\in\dR_{>0}.
      \]

  \item For every $u\in\tV_E^{(\infty)}$, denote by $\Pi_u^{[N]}$ the (complex irreducible) principal series representation of $\GL_N(E_u)$ given by the characters $(\arg^{1-N},\arg^{3-N},\dots,\arg^{N-3},\arg^{N-1})$, where $\arg\colon\dC^\times\to\dC^\times$ is the \emph{argument character} defined by the formula $\arg(z)\coloneqq z/\sqrt{z\ol{z}}$.

  \item A \emph{relevant $\dL$-representation} of $\GL_N(\dA_E^\infty)$ is a representation $\Pi$ with coefficients in $\dL$ satisfying that for every embedding $\iota\colon\dL\to\dC$,
      \begin{align}\label{eq:relevant}
      \Pi^{(\iota)}\coloneqq\(\otimes_{u\in\tV_E^{(\infty)}}\Pi_u^{[N]}\)\otimes\iota\Pi
      \end{align}
      is a hermitian isobaric automorphic representation of $\GL_N(\dA_E)$ \cite{BPLZZ}*{Definition~1.5}.\footnote{In this case, it is equivalent to saying that $\Pi^{(\iota)}$ is an isobaric sum of mutually non-isomorphic conjugate self-dual cuspidal automorphic representations.}
\end{enumerate}
\end{definition}

\begin{remark}\label{re:relevant}
We have the following remarks concerning Definition \ref{de:relevant}.
\begin{enumerate}
  \item A relevant $\dL$-representation is absolutely irreducible.

  \item A relevant $\dL$-representation is defined over a number field contained in $\dL$.

  \item It suffices to check the condition for one embedding $\iota$ in Definition \ref{de:relevant}(3).

  \item For a relevant $\dL$-representation $\Pi$, the number of isobaric summands of \eqref{eq:relevant} is independent of $\iota$, which we denote as $d(\Pi)$, which is an integer between $1$ and $N$.

  \item If $\Pi$ is a relevant $\dL$-representation, then for every embedding $\iota\colon\dL\to\dC$, $\iota\Pi_u$ is tempered for every $u\in\tV_E^\fin$ \cite{Car12}*{Theorem~1.2}.
\end{enumerate}
\end{remark}

Fix a positive integer $n$ throughout the article. Let $\Pi_n$ and $\Pi_{n+1}$ be two relevant $\dL$-representations of $\GL_n(\dA_E^\infty)$ and $\GL_{n+1}(\dA_E^\infty)$, respectively. Write $\Pi\coloneqq\Pi_n\boxtimes\Pi_{n+1}$ as a representation of $\GL_n(\dA_E^\infty)\times\GL_{n+1}(\dA_E^\infty)$. By the solution of the local Gan--Gross--Prasad conjecture \cite{GGP12} by \cites{BP15,BP16}, we know that for every finite place $v$ of $F$, there exists a pair $(V_{n,v},\pi_v)$, unique up to isomorphism, in which
\begin{itemize}
  \item $V_{n,v}$ is a hermitian space over $E_v$ of rank $n$;

  \item $\pi_v$ is an irreducible admissible representation of $G_v(F_v)$ with coefficients in $\dL$ with $\Pi_v$ as its base change, satisfying $\Hom_{H_v(F_v)}(\pi_v,\dL)\neq 0$.\footnote{Indeed, one can find an $\dL$-linear model of $\pi_v$ in the space of $\dL$-valued locally constant functions on $H_v(F_v)\backslash G_v(F_v)$, using the multiplicity one property.}
\end{itemize}
Here, we have put $G_v\coloneqq\rU(V_{n,v})\times\rU(V_{n+1,v})$ with $V_{n+1,v}\coloneqq V_{n,v}\oplus E_v\cdot\te$, and denoted by $H_v\subseteq G_v$ the graph of the natural embedding $\rU(V_{n,v})\hookrightarrow\rU(V_{n+1,v})$ by realizing the source as the stabilizer of $\te$ in the target.

\begin{definition}
Put
\[
\epsilon(\Pi)\coloneqq\prod_{v\in\tV_F^\fin}\eta_{E/F}\(\det V_{n,v}\) \in \{\pm 1\}.
\]
We say that $\Pi$ is a \emph{coherent} (resp.\ \emph{incoherent}) if $\epsilon(\Pi)$ equals $1$ (resp.\ $-1$).
\end{definition}

Now suppose that $\chi\colon\Gamma_\tP\to\dL_\chi^\times$ is a finite character. Then we have a new pair of relevant $\dL_\chi$-representations $\Pi_n\otimes\wt\chi$ and $\Pi_{n+1}$ of $\GL_n(\dA_E^\infty)$ and $\GL_{n+1}(\dA_E^\infty)$, respectively. Put
\[
\Pi_\chi\coloneqq(\Pi_n\otimes\wt\chi)\boxtimes\Pi_{n+1}=\Pi\otimes(\wt\chi\circ\det\boxtimes 1)
\]
as a representation of $\GL_n(\dA_E^\infty)\times\GL_{n+1}(\dA_E^\infty)$. Since $\Gamma_\tP$ is torsion free, $\chi_v$ is trivial for every finite place $v$ of $F$ that is nonsplit in $E$. It follows that for every finite place $v$ of $F$, the corresponding pair for $\Pi_\chi$ is $(V_{n,v},(\pi_v)_{\chi_v})$, where $(\pi_v)_{\chi_v}\coloneqq\pi_v\otimes(\chi_v\circ\det\boxtimes 1)$. In particular, $\epsilon(\Pi)=\epsilon(\Pi_\chi)$.

\begin{notation}\label{no:galois}
For $N=n,n+1$, we denote by
\[
\rho_{\Pi_N}\colon\Gal(\ol\dQ/E)\to\GL_N(\dL)
\]
the Galois representation associated with $\Pi_N$ satisfying $\rho_{\Pi_N}^\vee\simeq\rho_{\Pi_N}^\tc(N-1)$.\footnote{Strictly speaking, $\rho_{\Pi_N}$ can a priori only be defined over a finite extension $\dL'$ of $\dL$ (although it has traces in $\dL$). However, enlarging $\dL$ to a finite extension will do no harm to our discussion.} We propose the following condition which will be referred later:
\begin{itemize}
  \item[($*$)] For every irreducible $\Gal(\ol\dQ/E)$-subrepresentations $\rho_n$ and $\rho_{n+1}$ of $\rho_{\Pi_n}$ and $\rho_{\Pi_{n+1}}$, respectively, $\rho_n\otimes\rho_{n+1}$ is absolutely irreducible.
\end{itemize}
\end{notation}

In the rest of this section, we review some facts about local matrix coefficient integrals appearing in the refined Gan--Gross--Prasad conjecture. Take a finite place $v$ of $F$ and suppress it in the subscripts from the notation below. In particular, $E$ is a quadratic \'{e}tale $F$-algebra.

The lemma below will be used later.

\begin{lem}\label{le:rational}
Let $r,r_1,r_2$ be positive integers.
\begin{enumerate}
  \item Let $\Pi_1$ and $\Pi_2$ be tempered (complex) irreducible admissible representations of $\GL_{r_1}(E)$ and $\GL_{r_2}(E)$, respectively. Then for every automorphism $\sigma$ of $\dC$,
      \[
      L(s,\sigma\Pi_1\times\sigma\Pi_2)=\sigma
      L(s,\Pi_1\times\Pi_2)
      \]
      holds for every $s\in\dN+1$ (resp.\ $s\in\dN+\frac{1}{2}$) when $r_1$ and $r_2$ have the same parity (resp.\ different parities).

  \item Let $\Pi$ be a tempered (complex) irreducible admissible representations of $\GL_r(E)$ and $\epsilon\in\{\pm1\}$. Then for every automorphism $\sigma$ of $\dC$,
      \[
      L(s,\sigma\Pi,\As^\epsilon)=\sigma L(s,\Pi,\As^\epsilon)
      \]
      holds for every $s\in\dN+1$.
\end{enumerate}
\end{lem}

\begin{proof}
Denote by $e_r$ the vector $(0,\dots,0,1)$ with $r-1$ zeros.

For (1), without lost of generality, we may assume $r_1\leq r_2$ and $E$ a field. Regard $\GL_{r_1}$ as a subgroup of $\GL_{r_2}$ via the first $r_1$ coordinates. Let $U_{r_1}$ and $U_{r_2}$ be the standard upper-triangular maximal unipotent subgroups of $\GL_{r_1}$ and $\GL_{r_2}$ respectively. Choose two generic characters $\psi_1\colon U_{r_1}(E)\to\dC^\times$ and $\psi_2\colon U_{r_2}(E)\to\dC^\times$ satisfying $\psi_1(u)\psi_2(u)=1$ for $u\in U_{r_1}(E)$. For $\alpha=1,2$, denote by $\cW(\Pi_\alpha)_{\psi_\alpha}$ the Whittaker model of $\Pi_\alpha$ with respect to $(U_{r_\alpha},\psi_\alpha)$. Fix a rational Haar measure on $U_{r_1}(E)\backslash\GL_{r_1}(E)$. Denote by $\cF(E^{r_1})$ the space of Schwartz (resp.\ constant) functions on $E^{r_1}$ when $r_1=r_2$ (resp.\ $r_1<r_2$). Let $q$ be the residue cardinality of $E$.

For $W_1\in\cW(\Pi_1)_{\psi_1}$, $W_2\in\cW(\Pi_2)_{\psi_2}$, and $\Phi\in\cF(E^{r_1})$, the integral
\[
Z(s,W_1,W_2,\Phi)\coloneqq\int_{U_{r_1}(E)\backslash\GL_{r_1}(E)}W_1(g)W_2(g)\Phi(e_{r_1}g)|\det g|_F^{s-\frac{r_2-r_1}{2}}\rd g
\]
is absolutely convergent for $\RE s>0$ and has a meromorphic extension to an element in $\dC(q^{-s})$. By the local Rankin--Selberg theory \cite{JPSS83}, there exists a unique element $P_{\Pi_1\times\Pi_2}(X)\in\dC[X]$ satisfying $P_{\Pi_1,\Pi_2}(0)=1$ such that $\{Z(s,W_1,W_2,\Phi)\res W_1,W_2,\Phi\}$ is a $\dC[q^s,q^{-s}]$-submodule of $\dC(q^{-s})$ generated by $P_{\Pi_1,\Pi_2}(q^{-s})^{-1}$, which is nothing but $L(s,\Pi_1\times\Pi_2)$. In particular, $\{Z(s,W_1,W_2,\Phi)\res W_1,W_2,\Phi\}$ is independent of the choices of $\psi_1$ and $\psi_2$. Now let $\sigma$ be an automorphism of $\dC$. We have $\sigma W_\alpha\in\cW(\sigma\Pi_\alpha)_{\sigma\psi_\alpha}$ for $\alpha=1,2$, and moreover
\[
Z(s,\sigma W_1,\sigma W_2,\sigma\Phi)=\sigma Z(s,W_1,W_2,\Phi)
\]
as absolutely convergent integrals for $s\in\dN+1$ (resp.\ $s\in\dN+\frac{1}{2}$) when $r_1$ and $r_2$ have the same parity (resp.\ different parities). Since elements in $\dC(q^{-s})$ are uniquely determined by their values at an arbitrary infinite set of $s$, it follows easily that $L(s,\sigma\Pi_1\times\sigma\Pi_2)=\sigma L(s,\Pi_1\times\Pi_2)$ for the same set of $s$.

For (2), it is similar to (1) by considering the integral
\[
Z(s,W,\Phi)\coloneqq\int_{U_r(F)\backslash\GL_r(F)}W(g)\Phi(e_rg)\eta_{E/F}(\det g)^{\frac{1-\epsilon}{2}}|\det g|_F^s\rd g
\]
for $W\in\cW(\Pi)_\psi$ and $\Phi\in\cS(F^r)$.

The lemma is proved.
\end{proof}

Choose a rational Haar measure $\rd h$ on $H(F)$. Let $\chi\colon E^{\times-}\to\dL_\chi^\times$ be a finite character.

\begin{definition}\label{de:rational}
Lemma \ref{le:rational} provides us with the following definitions.
\begin{enumerate}
  \item We define $L(\tfrac{1}{2},(\Pi_n\otimes\wt\chi)\times\Pi_{n+1})\in\dL_\chi$ to be the unique element such that for every $\iota\colon\dL_\chi\to\dC$, $\iota L(\tfrac{1}{2},(\Pi_n\otimes\wt\chi)\times\Pi_{n+1})=L(\tfrac{1}{2},\iota(\Pi_n\otimes\wt\chi)\times\iota\Pi_{n+1})$.

  \item For $N=n,n+1$, we define $L(1,\Pi_N,\As^{(-1)^N})\in\dL^\times$ to be the unique element such that for every $\iota\colon\dL\to\dC$, $\iota L(1,\Pi_N,\As^{(-1)^N})=L(1,\iota\Pi_N,\As^{(-1)^N})$.
\end{enumerate}

\end{definition}

Let $\iota\colon\dL_\chi\to\dC$ be an embedding. For $\varphi\in\pi$ and $\varphi^\vee\in\pi^\vee$, the integral
\[
\alpha_\iota^\chi(\varphi,\varphi^\vee)\coloneqq
\int_{H(F)}\iota\(\langle\varphi^\vee,\pi(h)\varphi\rangle_\pi\cdot\chi(\det h)\) \rd h
\]
is absolutely convergent \cite{Har14} since $\iota\pi$ is tempered, hence defines an element in the $\dL_\chi$-vector space
\[
\alpha_\iota^\chi\in\Hom_{H(F)\times H(F)}\(\pi_\chi\boxtimes(\pi_\chi)^\vee,\dC\),
\]
where we regard $\dC$ as an $\dL_\chi$-vector space via $\iota$. We suppress $\chi$ in the notation when it is the trivial character.

\begin{lem}\label{le:matrix0}
Let $\chi\colon E^{\times-}\to\dL_\chi^\times$ be a finite character. For $\varphi\in\pi$ and $\varphi^\vee\in\pi^\vee$, there exists a unique element $\alpha^\chi(\varphi,\varphi^\vee)\in\dL_\chi$ such that for every embedding $\iota\colon\dL_\chi\to\dC$,
\[
\iota\alpha^\chi(\varphi,\varphi^\vee)=\alpha_\iota^\chi(\varphi,\varphi^\vee)
\]
holds.
\end{lem}

\begin{proof}
By Casselman's lemma \cite{Cas}*{Proposition~4.2.3~\&~Theorem~4.3.3} and the argument in \cite{Wal03}*{\S{I.2}}, the absolutely convergent integral defining $\alpha_\iota^\chi$ is Galois equivariant, namely, for every automorphism $\sigma$ of $\dC$,
\[
\int_{H(F)}\sigma\iota\(\langle\varphi^\vee,\pi(h)\varphi\rangle_\pi\cdot\chi(\det h)\) \rd h
=\sigma\int_{H(F)}\iota\(\langle\varphi^\vee,\pi(h)\varphi\rangle_\pi\cdot\chi(\det h)\) \rd h
\]
holds. The lemma then follows.
\end{proof}

\begin{lem}\label{le:matrix1}
Assume $\chi$ trivial if $E$ is a field. Then the $\dL_\chi$-vector space
\[
\Hom_{H(F)\times H(F)}\(\pi_\chi\boxtimes(\pi_\chi)^\vee,\dL_\chi\)
\]
is one-dimensional of which $\alpha^\chi$ is a basis.
\end{lem}

\begin{proof}
Without lost of generality, we may assume $\dL_\chi=\dC$. Then this is proved in \cite{BP16} or, in a more general setting, \cite{SV17}.
\end{proof}

\begin{lem}\label{le:matrix2}
Suppose that $E/F$ is unramified; $\chi$ is unramified; and $V_n$ admits an integrally self-dual lattice $L_n$ such that both $\varphi$ and $\varphi^\vee$ are fixed by $K_n\times K_{n+1}$, where $K_n$ and $K_{n+1}$ are the stabilizers of $L_n$ and $L_n\oplus O_E\cdot\te$ respectively. Then
\[
\alpha^\chi(\varphi,\varphi^\vee)=\langle\varphi^\vee,\varphi\rangle_\pi\cdot\vol(K_n,\rd h)\cdot
\frac{\Delta_{n+1}\cdot L(\tfrac{1}{2},(\Pi_n\otimes\wt\chi)\times\Pi_{n+1})}
{L(1,\Pi_n,\As^{(-1)^n})L(1,\Pi_{n+1},\As^{(-1)^{n+1}})}.
\]
\end{lem}

\begin{proof}
By Definition \ref{de:rational} and Lemma \ref{le:matrix0}, we may assume $\dL_\chi=\dC$. Then this is \cite{Har14}*{Theorem~2.12}.
\end{proof}

Now suppose that $E\simeq F\times F$. Choose a Borel subgroup $B$ of $H$ such that $B_H\coloneqq B\cap H$ is a Borel subgroup of $H$, and a generic character $\psi$ of $U(F)$ that is trivial on $(U\cap H)(F)$, where $U$ denotes the unipotent radical of $B$. For every embedding $\iota\colon\dL\to\dC$, denote by $\cW(\iota\pi)_\psi$ and $\cW(\iota\pi^\vee)_{\psi^{-1}}$ the Whittaker models of $\iota\pi$ and $\iota\pi^\vee$ with respect to $(U,\psi)$ and $(U,\psi^{-1})$, respectively. Define a pairing $\vartheta\colon\cW(\iota\pi^\vee)_{\psi^{-1}}\times\cW(\iota\pi)_\psi\to\dC$ by the formula
\[
\vartheta(W^\vee,W)\coloneqq\int_{U(F)\backslash Q(F)}W^\vee(h)W(h)\rd h
\]
where $Q$ is a mirabolic subgroup of $G$ containing $U$.

\begin{lem}\label{le:matrix3}
There exists a positive rational constant $c$ depending only on the rational Haar measures on $U(F)\backslash Q(F)$, $H(F)$, and $(U\cap H)(F)$, such that for every finite character $\chi\colon E^{\times-}\to\dL_\chi^\times$, every $\iota\colon\dL_\chi\to\dC$, and every $(W,W^\vee)\in\cW(\iota\pi)_\psi\times\cW(\iota\pi^\vee)_{\psi^{-1}}$,
\begin{align*}
&\int_{H(F)}\vartheta(W^\vee,\pi(h)W)\cdot\iota\chi(\det h)\rd h \\
&=c\(\int_{(U\cap H)(F)\backslash H(F)}W(h)\cdot\iota\chi(\det h)\rd h\)
\(\int_{(U\cap H)(F)\backslash H(F)}W^\vee(h)\cdot\iota\chi(\det h)^{-1}\rd h\)
\end{align*}
in which both sides are absolutely convergent.
\end{lem}

\begin{proof}
This is \cite{Zha14}*{Proposition~4.10}.
\end{proof}

\begin{proposition}\label{pr:matrix}
Suppose that $E\simeq F\times F$. There exist elements $\varphi\in\pi$ and $\varphi^\vee\in\pi^\vee$ such that for every finite unramified character $\chi\colon E^{\times-}\to\dL_\chi^\times$,
\[
\alpha^\chi(\varphi,\varphi^\vee)=\frac{\Delta_{n+1}\cdot L(\tfrac{1}{2},(\Pi_n\otimes\wt\chi)\times\Pi_{n+1})}
{L(1,\Pi_n,\As^{(-1)^n})L(1,\Pi_{n+1},\As^{(-1)^{n+1}})}
\]
holds.
\end{proposition}

\begin{proof}
Take an embedding $\iota_0\colon\dL\to\dC$. Choose $G(F)$-equivariant isomorphisms $i_1\colon\pi\otimes_{\dL,\iota_0}\dC\xrightarrow\sim\cW(\iota_0\pi)_\psi$ and $i_2\in\pi^\vee\otimes_{\dL,\iota_0}\dC\xrightarrow\sim\cW(\iota_0\pi^\vee)_{\psi^{-1}}$ such that $\iota_0\langle\varphi^\vee,\varphi\rangle_\pi=\vartheta(i_2\varphi^\vee,i_1\varphi)$ holds for all $\varphi\in\pi$ and $\varphi^\vee\in\pi^\vee$. By the definition of Rankin--Selberg $L$-factors \cite{JPSS83}, there exist elements $W\in i_1(\pi)$ and $W^\vee\in i_2(\pi^\vee)$ such that
\begin{align*}
\int_{(U\cap H)(F)\backslash H(F)}W(h)\cdot|\det h|_F^s\rd h &= c_1 L(\tfrac{1}{2}+s,\iota_0\pi_n\times\iota_0\pi_{n+1}), \\
\int_{(U\cap H)(F)\backslash H(F)}W^\vee(h)\cdot|\det h|_F^s\rd h &=c_2 L(\tfrac{1}{2}+s,\iota_0\pi_n^\vee\times\iota_0\pi_{n+1}^\vee),
\end{align*}
hold as meromorphic functions on $\dC$ for some constants $c_1,c_2\in\dC^\times$. Moreover, the integrals on the left-hand side are absolutely convergent when $\RE s\geq 0$. Put $\varphi=i_1^{-1}W$ and $\varphi^\vee=i_2^{-1}W^\vee$. By Lemma \ref{le:matrix3}, we have
\[
\alpha_{\iota_0}(\varphi,\varphi^\vee)=cc_1c_2\cdot L(\tfrac{1}{2},\iota_0\Pi_n\times\iota_0\Pi_{n+1}).
\]
Since $\alpha_{\iota_0}(\varphi,\varphi^\vee)\in\iota_0\dL$ and $L(\tfrac{1}{2},\iota_0\Pi_n\times\iota_0\Pi_{n+1})\in\iota_0\dL^\times$, we have $cc_1c_2\in\iota_0\dL^\times$. By Lemma \ref{le:rational}(2), after rescaling, we may assume
\[
cc_1c_2=\frac{\Delta_{n+1}}{L(1,\iota_0\Pi_n,\As^{(-1)^n})L(1,\iota_0\Pi_{n+1},\As^{(-1)^{n+1}})}.
\]
Then, again by Lemma \ref{le:matrix3}, we know that for every finite unramified character $\chi\colon E^{\times-}\to\dL_\chi^\times$ and every $\iota\colon\dL_\chi\to\dC$ extending $\iota_0$,
\[
\alpha_\iota^\chi(\varphi,\varphi^\vee)=\frac{\Delta_{n+1}\cdot L(\tfrac{1}{2},\iota(\Pi_n\otimes\wt\chi)\times\iota\Pi_{n+1})}
{L(1,\iota\Pi_n,\As^{(-1)^n})L(1,\iota\Pi_{n+1},\As^{(-1)^{n+1}})}
\]
holds. Finally, by Definition \ref{de:rational} and Lemma \ref{le:matrix0}, the above identity holds for every embedding $\iota\colon\dL_\chi\to\dC$.

The proposition is proved.
\end{proof}

\section{$p$-adic measure}

From now on, we assume that $\dL$ is a finite extension of $\dQ_p$, and denote by $\dL^\circ$ its ring of integers.

\begin{notation}\label{no:dagger}
For every quotient group $\Gamma$ of $\Gamma_\tP$, put
\[
\dL[[\Gamma]]^\circ\coloneqq\dL^\circ[[\Gamma]]\otimes_{\dL^\circ}\dL
\]
and denote by $-^\dag$ the involution on $\dL^\circ[[\Gamma]]$ or $\dL[[\Gamma]]^\circ$ induced by the inverse map on $\Gamma$.
\end{notation}

For every $v\in\tP$ and every positive integer $\fc_v$, we put
\[
\fU_{\fc_v}\coloneqq\left\{(x,x^{-1})\res x\in 1+\fp_v^{\fc_v}\right\}\subseteq E_v^\times.
\]
For a tuple $\fc=(\fc_v)_{v\in\tP}$ of positive integers indexed by $\tP$, we put
\[
\fU_\fc\coloneqq\prod_{v\in\tP}\fU_{\fc_v}\subseteq\prod_{v\in\tP}E_v^\times.
\]
Then $\Nm_{E/F}^-$ induces a homomorphism
\[
\fU_\fc\to\Gamma_\tP
\]
whose kernel is finite, and trivial when $|\fc|\coloneqq\sum_v\fc_v$ is large enough. We define a \emph{$\fc$-cell} of $\Gamma_\tP$ to be a $\fU_\fc$-orbit in $\Gamma_\tP$. Then for every given $\fc$, all $\fc$-cells give a disjoint cover of $\Gamma_\tP$ by open compact subsets.

In what follows, we will frequently choose a collection $\varpi=(\varpi_v)_{v\in\tP}$ of uniformizers of $F_v$ for $v\in\tP$. For $\fc$ as above, we put
\[
\varpi^\fc\coloneqq(\varpi_v^{\fc_v})_{v\in\tP}\in\prod_{v\in\tP}(O_{F_v}\cap F_v^\times).
\]

In this article, we will encounter measures on $\Gamma_\tP$ valued in a finite-dimensional $\dL$-vector space $\dV$, whose definition we now review. A \emph{distribution} on $\Gamma_\tP$ valued in $\dV$ is an assignment
\[
\bmu\colon\{\text{open compact subsets of $\Gamma_\tP$}\}\to\dV
\]
that is additive. It is clear that in the above definition, we may replace the source by the set of all $\fc$-cells of $\Gamma_\tP$ for all $\fc$. We say that the distribution $\bmu$ is a \emph{measure} if it is bounded, namely, there exists an $\dL^\circ$-lattice $\dV^\circ$ of $\dV$ such that the range of $\bmu$ is contained in $\dV^\circ$.

For a measure $\bmu$ valued in $\dV$, we may evaluate it on every continuous character $\chi\colon\Gamma_\tP\to\dL_\chi^\times$ (for a $p$-adic field extension $\dL_\chi/\dL$) by the formula
\[
\bmu(\chi)\coloneqq\int_{\Gamma_\tP}\chi\rd\bmu\in\dV\otimes_\dL\dL_\chi
\]
(as the limit of Riemann sums over $\fc$-cells for all $\fc$). By evaluating at all finite characters of $\Gamma_\tP$, we obtain a map
\[
\{\text{measures on $\Gamma_\tP$ valued in $\dV$}\}\to \dL[[\Gamma_\tP]]^\circ\otimes_\dL\dV,
\]
which is well-known to be a bijection. As consequences, we have
\begin{itemize}
  \item A measure is determined by its values on the set of finite characters of $\Gamma_\tP$ that are ramified at every place above $\tP$.

  \item When $\dV$ is a (commutative) $\dL$-algebra, the set of measures on $\Gamma_\tP$ valued in $\dV$ is naturally a (commutative) $\dL$-algebra.
\end{itemize}

\section{Ordinary condition}
\label{ss:ordinary}

Take an element $v\in\tV_F^{(p)}\cap\tV_F^\spl$ and suppress it in the subscripts from the notation below. In particular, $E\simeq F\times F$.

\begin{definition}\label{de:standard}
An isomorphism $G\simeq\GL_{n,F}\times\GL_{n+1,F}$ is called \emph{standard} if the subgroup $H$ is identified with $(h,\diag(h,1))$ for $h\in\GL_{n,F}$.
\end{definition}

We fix a standard isomorphism $G\simeq\GL_{n,F}\times\GL_{n+1,F}$ (which exists). For $N=n,n+1$, denote by $B_N$ and $U_N$ the upper-triangular Borel and unipotent subgroups of $\GL_N$, respectively; for every positive integer $r$, denote by $I_N^{(r)}$ the inverse image of $B_N(O_F/\fp^r)$ under the reduction map $\GL_N(O_F)\to\GL_N(O_F/\fp^r)$. Put $B\coloneqq B_n\times B_{n+1}$, $U\coloneqq U_n\times U_{n+1}$, and $I^{(r)}\coloneqq I_n^{(r)}\times I_{n+1}^{(r)}$. We also write $\pi=\pi_n\boxtimes\pi_{n+1}$ in which both $\pi_n$ and $\pi_{n+1}$ have coefficients in $\dL$ with base change $\Pi_n$ and $\Pi_{n+1}$, respectively.

For $N=n,n+1$ and a tuple $\mu=(\mu_1,\dots,\mu_N)$ of admissible characters $F^\times\to\dL^\times$, we have an induced character $\mu^\natural$ of $(F^\times)^N$ given by
\[
\mu^\natural(x_1,\dots,x_N)=\prod_{i=1}^N\mu_i(x_i)|x_i|_F^{N-i}
\]
hence a character of $B_N(F)$ by inflation, and the unnormalized principal series
\[
\cI(\mu)\coloneqq\left\{f\colon\GL_N(F)\to\dL\text{ locally constant}\left| f(bg)=\mu^\natural(b)f(g),\forall b\in B_N(F),g\in\GL_N(F)\right.\right\}
\]
as an admissible representation of $\GL_N(F)$ via right translation.

\begin{definition}\label{de:ordinary}
For $N=n,n+1$, we say that $\Pi_N$ is \emph{ordinary} if there exists a (unique) tuple $\mu=(\mu_1,\dots,\mu_N)$ of admissible characters $F^\times\to\dL^\times$ satisfying $|x|_F^{i-1}\mu_i(x)\in \dL^{\circ\times}$ for $1\leq i\leq N$ and every $x\in F^\times$ such that $\pi_N$ is a subrepresentation of $\cI(\mu)$; we say that $\Pi_N$ is \emph{semi-stably ordinary} if furthermore $\mu_i$ are all unramified.

We say that $\Pi$ is (semi-stably) ordinary if both $\Pi_n$ and $\Pi_{n+1}$ are.
\end{definition}

\begin{remark}\label{re:ordinary}
Note that if $\pi_N$ satisfies the property in Definition \ref{de:ordinary}, then so does $\pi_N^\vee$ with respect to the tuple $\check\mu\coloneqq(|\;|_F^{1-N}\mu_N^{-1},\dots,|\;|_F^{1-N}\mu_1^{-1})$.
\end{remark}

For $N=n,n+1$ and an element $x\in O_F\cap F^\times$, we put $[x]_N\coloneqq\diag(x^{N-1},\dots,x,1)\in\GL_N(F)$ and define an operator $\tV_N^x$ on $\pi_N^{U_N(O_F)}$ as
\begin{align*}
\tV_N^x&\coloneqq\sum_{u\in U_N(O_F)/(U_N(O_F)\cap [x]_N U_N(O_F)[x]_N^{-1})}\pi_N(u [x]_N).
\end{align*}

\begin{lem}\label{le:ordinary}
Suppose that $\Pi_N$ is ordinary for some $N\in\{n,n+1\}$. There exists a unique up to scalar nonzero element $\varphi_N\in\pi_N^{U_N(O_F)}$ satisfying that
\[
\tV_N^x\varphi_N=\(\prod_{m=1}^{N-1}\prod_{i=1}^m|x|_F^{i-1}\mu_i(x)\)\varphi_N
\]
holds for every $x\in O_F\cap F^\times$. Here, $\mu$ is the tuple of characters in Definition \ref{de:ordinary}.
\end{lem}

\begin{proof}
Since $\mu_1\boxtimes\cdots\boxtimes\mu_N$ is a regular character of $B_N(F)/U_N(F)$, the Jacquet module (with respect to $B_N$) $\cI(\mu)_{B_N}$ of $\cI(\mu)$ is a direct sum of $N$ distinct characters of $B_N(F)/U_N(F)$. By \cite{Hid98}*{Corollary~5.5}, it suffices to show that $\pi_{B_N}$ contains the character $\mu^\natural$, which follows from the Frobenius reciprocity law. The lemma is proved.
\end{proof}

\begin{definition}\label{de:ordinary1}
We call the one-dimensional $\dL$-subspace of $\pi_N^{U_N(O_F)}$ generated by $\varphi_N$ in Lemma \ref{le:ordinary} the \emph{ordinary line} of $\pi_N$, and a nonzero element of it an \emph{ordinary vector}. When $\Pi$ is ordinary, we have the obvious notion of ordinary line and ordinary vectors, which are contained in $\pi^{U(O_F)}=\pi_n^{U_n(O_F)}\otimes_\dL\pi_{n+1}^{U_{n+1}(O_F)}$.
\end{definition}

\begin{remark}
The notion of being (semi-stably) ordinary and the characters in Definition \ref{de:ordinary} are intrinsic, namely, they do not depend on the choice of the standard isomorphism $G\simeq\GL_{n,F}\times\GL_{n+1,F}$. However, the notion of ordinary line and vectors clearly depends on such choice.
\end{remark}

\begin{lem}\label{le:ordinary2}
Suppose that $\Pi_N$ is ordinary for some $N\in\{n,n+1\}$. For ordinary vectors $\varphi_N\in\pi_N^{U_N(O_F)}$ and $\varphi^\vee_N\in(\pi_N^\vee)^{U_N(O_F)}$, we have $\langle\varphi^\vee_N,\varphi_N\rangle_{\pi_N}\neq 0$.
\end{lem}

\begin{proof}
Fix a uniformizer $\varpi$ of $F$. Denote by $(\pi_N)^{\r{ss}}$ the invertible part of $\pi_N^{U_N(O_F)}$ with respect to the operator $\tV^\varpi_N$ (see \cite{Hid98}*{\S5.2}). Then $\langle\;,\;\rangle_{\pi_N}$ restricts to a perfect pairing on $(\pi_N^\vee)^{\r{ss}}\times(\pi_N)^{\r{ss}}$. Moreover, by \cite{Hid98}*{Proposition~5.4~\&~Corollary~5.5}, there exists a group $\fS$ consisting of permutations on $\{1,\dots,N\}$ such that $(\pi_N)^{\r{ss}}$ and $(\pi_N^\vee)^{\r{ss}}$ are spanned by $\{\varphi_{N,\sigma}\res\sigma\in\fS\}$ and $\{\varphi^\vee_{N,\sigma}\res\sigma\in\fS\}$ in which $\varphi_{N,\sigma}$ and $\varphi^\vee_{N,\sigma}$ are nonzero eigenvectors of $\tV^\varpi_N$ with eigenvalues
\[
\prod_{m=1}^{N-1}\prod_{i=1}^m|\varpi|_F^{i-1}\mu_{\sigma(i)}(\varpi),\quad
\prod_{m=1}^{N-1}\prod_{i=1}^m|\varpi|_F^{i-1}\check\mu_{\sigma(i)}(\varpi),
\]
respectively. Now since $\langle\tV^\varpi_N\;,\;\rangle_{\pi_N}=\langle\;,\varpi^{1-N}\tV^\varpi_N\;\rangle_{\pi_N}$, it follows that $\langle\varphi^\vee_{N,\sigma_2},\varphi_{N,\sigma_1}\rangle_{\pi_N}=0$ if $\sigma_1\neq\sigma_2$. In particular, $\langle\varphi^\vee_{N,1},\varphi_{N,1}\rangle_{\pi_N}\neq 0$, which proves the lemma.
\end{proof}

\begin{lem}\label{le:ordinary1}
Suppose that $\Pi_N$ is ordinary for some $N\in\{n,n+1\}$. The following are equivalence:
\begin{enumerate}
  \item $\Pi_N$ is semi-stably ordinary.

  \item The ordinary line is contained in $\pi_N^{I_N^{(r)}}$ for some positive integer $r$.

  \item The ordinary line is contained in $\pi_N^{I_N^{(1)}}$.
\end{enumerate}
\end{lem}

\begin{proof}
The implication $(3)\Rightarrow(2)$ is trivial.

For $(2)\Rightarrow(1)$, let $\varphi_N$ be an ordinary vector fixed by $I_N^{(r)}$. Then it is also the eigenvector of the operators $\tV_{N,m}^\varpi$ in \cite{Jan15}*{Lemma~3.2} for $0\leq m\leq N$ (with respect to a uniformizer $\varpi$) with the eigenvalue $\prod_{i=1}^m|\varpi|_F^{i-1}\mu_i(\varpi)$. Now since $\varphi_N$ is fixed by $I_N^{(r)}$, the eigenvalues are independent of the choice of $\varpi$, which implies that $\mu_1,\mu_1\mu_2,\dots,\mu_1\cdots\mu_N$ are all unramified characters, that is, $\Pi_N$ is semi-stably ordinary.

For $(1)\Rightarrow(3)$, note that the above discussion also shows $(1)\Rightarrow(2)$, namely, the ordinary line is contained in $\pi_N^{I_N^{(r)}}$ for some $r\geq 1$. Then the statement follows from \cite{Jan15}*{Proposition~3.4} by noting that elements in $\pi_N^{I_N^{(r)}}$ that are annihilated by the Hecke polynomial of $\pi_N$ are already contained in $\pi_N^{I_N^{(1)}}$.
\end{proof}

\begin{notation}\label{no:ordinary}
Suppose that $\Pi$ is ordinary. Let $\mu$ and $\nu$ be the two tuples of characters for $\pi_n$ and $\pi_{n+1}$ in Definition \ref{de:ordinary}, respectively. For every $x\in O_F\cap F^\times$, put
\[
\omega^x(\pi)\coloneqq\(\prod_{m=1}^n\prod_{i=1}^m|x|_F^{i-1}\mu_i(x)\)\(\prod_{m=1}^n\prod_{i=1}^m|x|_F^{i-1}\nu_i(x)\)\in \dL^{\circ\times},
\]
and similarly for $\omega^x(\pi^\vee)$ (see Remark \ref{re:ordinary}). When $\Pi$ is semi-stably ordinary, we simply put $\omega(\pi)\coloneqq\omega^\varpi(\pi)$ and $\omega(\pi^\vee)\coloneqq\omega^\varpi(\pi^\vee)$ for some hence every uniformizer $\varpi$ of $F$, and put $\omega(\Pi)\coloneqq\omega(\pi)\omega(\pi^\vee)$.
\end{notation}

\begin{notation}\label{no:ordinary1}
We introduce two more notations.
\begin{enumerate}
  \item Put
    \[
    \xi\coloneqq
    \left(
      \begin{array}{cccc}
         &  & 1 & 1 \\
         & \iddots &  & \vdots \\
        1 &  &  & 1 \\
        0 & \cdots & 0 & 1 \\
      \end{array}
    \right)\in\GL_{n+1}(\dZ)
    \]
    as the element introduced on the top of \cite{Jan11}*{Page~23}.

  \item For $x\in O_F\cap F^\times$, put $[x]\coloneqq x[x]_n\in\GL_n(F)$, regarded as an element of $H(F)$ (in particular, as an element of $G(F)$, $[x]=(x[x]_n,[x]_{n+1})$).
\end{enumerate}
\end{notation}

\begin{lem}\label{le:ordinary3}
For every positive integer $\fc$, put $K_{\fc,H}\coloneqq((1_n,\xi)\cdot[\varpi^\fc])I^{(1)}((1_n,\xi)\cdot[\varpi^\fc])^{-1}\cap H(F)$. Then
\begin{enumerate}
  \item $\det(K_{\fc,H})=1+\fp^\fc$;

  \item $K_{\fc+1,H}$ is contained in $K_{\fc,H}$ of index $q^{\frac{n(n+1)(2n+1)}{6}}$.
\end{enumerate}
\end{lem}

\begin{proof}
It is clear that $K_{\fc,H}$ does not depend on the choice of $\varpi$. Let $h=(h_{ij})_{1\leq i,j\leq n}\in H(F)$ be an element written in the matrix form. Then a straightforward computation shows that the element $((1_n,\xi)\cdot[\varpi^\fc])^{-1}h((1_n,\xi)\cdot[\varpi^\fc])$ belongs to $I^{(1)}$ if and only if the following are satisfied:
\begin{itemize}
  \item $h_{ij}\in\fp^{|i-j|\fc}$ for $i\neq j$;

  \item $h_{i1}+\cdots+h_{in}-1\in\fp^{i\fc}$ for $1\leq i\leq n$.
\end{itemize}
Then (1) follows directly. For (2), the inclusion is clear; while for the index, we have
\[
[K_{\fc,H}:K_{\fc+1,H}]=q^{\sum_{i=1}^{n-1} i(n-i)}\cdot q^{\sum_{i=1}^{n-1} i(n-i)}\cdot q^{\sum_{i=1}^n i}
=q^{\frac{n(n+1)(2n+1)}{6}}.
\]
The lemma is proved.
\end{proof}

\begin{proposition}\label{pr:ordinary}
Suppose that $\Pi$ is semi-stably ordinary. There exists a unique element $\gamma(\Pi)\in\dL^\times$ depending only on $\Pi$ such that for every pair of ordinary vectors $\varphi\in\pi,\varphi^\vee\in\pi^\vee$, every positive integer $\fc$, and every finite character $\chi\colon E^{\times-}\to\dL_\chi^\times$ of conductor $\fp^\fc$,\footnote{We adopt the convention that when $q=2$, there is no character of conductor $\fp$. In particular, $\chi$ must be ramified here.}
\begin{align*}
&\frac{\gamma(\Pi)}{\langle\varphi^\vee,\varphi\rangle_\pi}\cdot
\alpha^\chi\(\pi((1_n,\xi)\cdot[\varpi^\fc])\varphi,\pi^\vee((1_n,\xi)\cdot[\varpi^\fc])\varphi^\vee\) \\
&=\(q^{-\frac{n(n+1)(2n+1)}{6}}\)^\fc\frac{\Delta_{n+1}}{L(1,\Pi_n,\As^{(-1)^n})L(1,\Pi_n+1,\As^{(-1)^{n+1}})},
\end{align*}
where we adopt the Haar measure on $H(F)$ that gives a hyperspecial maximal subgroup volume $1$.
\end{proposition}

\begin{proof}
Choose an embedding $\iota\colon\dL\to\dC$ and an additive character $\psi_F\colon F\to\dC^\times$ of conductor $O_F$. Let $\psi\colon U(F)\to\dC^\times$ be the character sending $((u_{ij})_{1\leq i,j\leq n},(v_{ij})_{1\leq i,j\leq n+1})\in U_n(F)\times U_{n+1}(F)=U(F)$ to $\psi_F(-u_{12}-\cdots-u_{n-1n}+v_{12}+\cdots+v_{nn+1})$, which is a generic character trivial on $(U\cap H)(F)$. Choose $G(F)$-equivariant isomorphisms $i_1\colon\pi\otimes_{\dL,\iota}\dC\xrightarrow\sim\cW(\iota\pi)_\psi$ and $i_2\in\pi^\vee\otimes_{\dL,\iota}\dC\xrightarrow\sim\cW(\iota\pi^\vee)_{\psi^{-1}}$ such that $\iota\langle\varphi^\vee,\varphi\rangle_\pi=\vartheta(i_2\varphi^\vee,i_1\varphi)$ holds for all $\varphi\in\pi$ and $\varphi^\vee\in\pi^\vee$. By Lemma \ref{le:matrix3}, there exists a constant $c\in\dQ^\times$ depending only on certain rational Haar measures, such that for every finite character $\chi\colon E^{-\times}\to\dC^\times$,
\[
\alpha^\chi(\varphi,\varphi^\vee)=c\(\int_{(U\cap H)(F)\backslash H(F)}(i_1\varphi)(h)\cdot\chi(\det h)\rd h\)
\(\int_{(U\cap H)(F)\backslash H(F)}(i_2\varphi^\vee)(h)\cdot\chi(\det h)^{-1}\rd h\)
\]
for every $\varphi\in\pi$ and $\varphi^\vee\in\pi^\vee$. Apply this formula to $\pi((1_n,\xi)\cdot[\varpi^\fc])\varphi$ and $\pi^\vee((1_n,\xi)\cdot[\varpi^\fc])\varphi^\vee$ with $\varphi$ and $\varphi^\vee$ ordinary vectors when $\chi$ has conductor $\fp^\fc$. By \cite{Jan11}*{Corollary~2.8}, we obtain
\begin{align}\label{eq:ordinary}
&\alpha^\chi\(\pi((1_n,\xi)\cdot[\varpi^\fc])\varphi,\pi^\vee((1_n,\xi)\cdot[\varpi^\fc])\varphi^\vee\) \notag \\
&=c\prod_{i=1}^n(1-q^{-i})^{-2}\cdot\(q^{-\frac{n(n+1)(n+2)}{3}}\)^\fc
\cdot G_{\psi_F}(\chi)^{\frac{n(n+1)}{2}}G_{\psi_F^{-1}}(\chi^{-1})^{\frac{n(n+1)}{2}}\cdot (i_1\varphi)(1)\cdot (i_2\varphi^\vee)(1),
\end{align}
where $G_{\psi_F^{\pm1}}(\chi^{\pm1})$ denotes the Gauss sums (the one in \cite{Jan15}*{\S2}) of $\chi^{\pm 1}$ with respect to $\psi_F^{\pm 1}$. As it is well-known that $G_{\psi_F}(\chi)G_{\psi_F^{-1}}(\chi^{-1})=q^\fc$, we have
\[
\eqref{eq:ordinary}=c\prod_{i=1}^n(1-q^{-i})^{-2}\cdot\(q^{-\frac{n(n+1)(2n+1)}{6}}\)^\fc
\cdot (i_1\varphi)(1)\cdot (i_2\varphi^\vee)(1).
\]
By \cite{KMS00}*{Proposition~4.12}, we have $(i_1\varphi)(1)\cdot (i_2\varphi^\vee)(1)\neq 0$. Put
\[
c'\coloneqq\frac{\prod_{i=1}^n(1-q^{-i})^2\cdot\langle\varphi^\vee,\varphi\rangle_\pi}{c\cdot (i_1\varphi)(1)\cdot (i_2\varphi^\vee)(1)},
\]
which belongs to $\dC^\times$ by Lemma \ref{le:ordinary2} and depends only on $\Pi$ and certain Haar measures. Then we have
\begin{align*}
\frac{c'}{\langle\varphi^\vee,\varphi\rangle_\pi}\cdot
\alpha^\chi\(\pi((1_n,\xi)\cdot[\varpi^\fc])\varphi,\pi^\vee((1_n,\xi)\cdot[\varpi^\fc])\varphi^\vee\)
=\(q^{-\frac{n(n+1)(2n+1)}{6}}\)^\fc
\end{align*}
for every $\fc\geq 1$ and every finite character $\chi\colon E^{-\times}\to\dC^\times$ of conductor $\fp^\fc$. Now it is clear that $c'$ depends only on $\Pi$. Finally, we may choose a ramified character $\chi$ that takes values in $\iota\dL^\times$, which implies that $c'\in\iota\dL^\times$. The proposition follows by taking
\[
\gamma(\Pi)\coloneqq(\iota^{-1}c')\frac{\Delta_{n+1}}{L(1,\Pi_n,\As^{(-1)^n})L(1,\Pi_{n+1},\As^{(-1)^{n+1}})}\in\dL^\times.
\]
\end{proof}

\section{Coherent anticyclotomic $p$-adic $L$-function}
\label{ss:coherent}

In this section, we take $\tP$ to be a subset of
\begin{align}\label{eq:ordinaryp}
\tP(\Pi)\coloneqq\left\{\left.v\in\tV_F^{(p)}\cap\tV_F^\spl\right|\text{$\Pi_v$ is semi-stably ordinary}\right\}.
\end{align}

We study the case where $\Pi=\Pi_n\boxtimes\Pi_{n+1}$ is coherent and aim to construct a $p$-adic $L$-function $\sL^0_\cE(\Pi)$ for every $\tP$-extension $\cE/E$ (Definition \ref{de:extension}). In the coherent case, there exists a totally positive definite hermitian space $V_n$ over $E$ of rank $n$, unique up to isomorphism, such that $V_{n,v}$ is the prescribed hermitian space from Section \ref{ss:ggp} for every $v\in\tV_F^\fin$. Put $V_{n+1}\coloneqq V_n\oplus E\cdot\te$. Put
\[
G\coloneqq\rU(V_n)\times\rU(V_{n+1}),\quad
\pi\coloneqq\otimes_{v\in\tV_F^\fin}\pi_v
\]
which is an irreducible admissible representation of $G(\dA_F^\infty)$ with coefficients in $\dL$.

Put
\[
\dV_\Pi\coloneqq\Hom_{G(\dA_F^\infty)}\(\pi,\cS\(G(F)\backslash G(\dA_F^\infty),\dL\)\),\quad
\dV_{\Pi^\vee}\coloneqq\Hom_{G(\dA_F^\infty)}\(\pi^\vee,\cS\(G(F)\backslash G(\dA_F^\infty),\dL\)\).
\]
Define $\sL^0_\cE(\Pi)=0$ when $\dV_\Pi=0$.

Now we assume $\dV_\Pi\neq 0$ hence $\dV_{\Pi^\vee}\neq 0$. Then by Arthur's multiplicity formula \cites{Mok15,KMSW}, we have $\dim_\dL\dV_\Pi=\dim_\dL\dV_{\Pi^\vee}=1$. Denote by $\cV,\cV^\vee\subseteq\cS\(G(F)\backslash G(\dA_F^\infty),\dL\)$ the unique irreducible $\dL[G(\dA_F^\infty)]$-submodules that are isomorphic to $\pi$ and $\pi^\vee$, respectively. Using the Petersson inner product with respect to the Tamagawa measure of $G(\dA_F)$, we have a canonical isomorphism $\pi\otimes_\dL\pi^\vee\simeq\cV\otimes_\dL\cV^\vee$.

Denote by $H\subseteq G$ the graph of the natural embedding $\rU(V_n)\hookrightarrow\rU(V_{n+1})$, and fix a decomposition $\rd h=\rd h_\infty\cdot\rd h_\tP\cdot\rd h^{\infty,\tP}$ of the Tamagawa measure on $H(\dA_F)$ such that the volume of $H(F_\infty)$ under $\rd h_\infty$ is $1$ and the volume of every hyperspecial maximal subgroup of $H(F_\tP)$ under $\rd h_\tP$ is $1$ (so that the measure $\rd h^{\infty,\tP}$ is rational).

For every $v\in\tP$, fix a standard isomorphism $G_v\simeq\GL_{n,F_v}\times\GL_{n+1,F_v}$ (Definition \ref{de:standard}). For every element $\varphi=\otimes_{v\in\tV_F^\fin}\varphi_v\in\cV$ with $\varphi_v\in\pi_v$ that is a (nonzero) ordinary vector for $v\in\tP$, we define a measure $\sP_\varphi\in\dL[[\Gamma_\tP]]^\circ$ as follows: For every subset $\Omega$ of $\Gamma_\tP$, we denote by $H_\Omega$ its inverse image under the determinant map $H(F)\backslash H(\dA_F^\infty)\to\Gamma_\tP$. For a $\fc$-cell $\Omega$ of $\Gamma_\tP$ for some tuple $\fc=(\fc_v)_{v\in\tP}$ of positive integers, we put
\[
\sP_\varphi(\Omega)\coloneqq\prod_{v\in\tP}\(\frac{q_v^{\frac{n(n+1)(2n+1)}{6}}}{\omega(\pi_v)}\)^{\fc_v}
\cdot\int_{H_\Omega}\varphi\(h^\infty\cdot(1_n,\xi)\cdot[\varpi^\fc]\)\rd h^\infty,
\]
in which $(1_n,\xi)\in\GL_n(F_\tP)\times\GL_{n+1}(F_\tP)=G(F_\tP)$ and $[\varpi^\fc]\in\GL_n(F_\tP)=H(F_\tP)$. It is clear that the above integral is a finite sum of elements in $\dL$ and is independent of the choice of $\varpi$ by Lemma \ref{le:ordinary1}.

\begin{proposition}
The assignment $\Omega\mapsto\sP_\varphi(\Omega)$ defines a measure $\sP_\varphi$ on $\Gamma_\tP$ valued in $\dL$; in other words, $\sP_\varphi$ is an element of $\dL[[\Gamma_\tP]]^\circ$.
\end{proposition}

\begin{proof}
The additivity of $\sP_\varphi$ follows from the same argument in the proof of \cite{Jan11}*{Theorem~4.4} (with a translation by $[\varpi^\fc]$). For the boundedness, Lemma \ref{le:ordinary3}(2) implies that there exists a positive integer $C$ such that for every tuple $\fc=(\fc_v)_{v\in\tP}$ of positive integers and every $\fc$-cell $\Omega$, the volume of $H_\Omega$ is bounded by
\[
C\cdot\prod_{v\in\tP}\(q_v^{-\frac{n(n+1)(2n+1)}{6}}\)^{\fc_v}.
\]
Thus, the boundedness of $\sP_\varphi$ follows from the fact that $\omega(\pi_v)\in\dL^{\circ\times}$ for $v\in\tP$ and that $\varphi$ is bounded.
\end{proof}

\begin{theorem}\label{th:function}
There exists a unique measure $\sL^0_{\cE_\tP}(\Pi)\in\dL[[\Gamma_\tP]]^\circ$ such that for every finite character $\chi\colon\Gamma_\tP\to\dL_\chi^\times$ of conductor $\prod\fp_v^{\fc_v}$ for a tuple $\fc=(\fc_v)_{v\in\tP}$ of positive integers indexed by $\tP$ and every embedding $\iota\colon\dL_\chi\to\dC$, we have
\[
\iota\sL^0_{\cE_\tP}(\Pi)(\chi)=\prod_{v\in\tP}\(\frac{q_v^{\frac{n(n+1)(2n+1)}{6}}}{\iota\omega(\Pi_v)}\)^{\fc_v}\cdot
\frac{\Delta_{n+1}\cdot L(\tfrac{1}{2},\iota(\Pi_n\otimes\wt\chi)\times\iota\Pi_{n+1})}
{2^{d(\Pi_n)+d(\Pi_{n+1})}\cdot L(1,\iota\Pi_n,\As^{(-1)^n})L(1,\iota\Pi_{n+1},\As^{(-1)^{n+1}})}.
\]
Here, $d(\Pi_N)$ for $N=n,n+1$ is introduced in Remark \ref{re:relevant}(4).\footnote{The evaluation of $\sL^0_{\cE_\tP}(\Pi)$ at general finite order characters will be carried out in \cite{LS}.}
\end{theorem}

The theorem holds even when $\dV_\Pi=0$ since in this case $L(\tfrac{1}{2},\iota(\Pi_n\otimes\wt\chi)\times\iota\Pi_{n+1})$ always vanishes; see Remark \ref{re:vanishing1}.

\begin{proof}
The uniqueness is clear. Now we show the existence. Since $\Gamma_\tP$ is torsion free, $\chi_v=1$ for every finite order character $\chi$ as in the theorem and every $v\in\tV_F^\fin\setminus\tV_F^\spl$.

Choose a pair $\varphi=\otimes_{v\in\tV_F^\fin}\varphi_v\in\cV$ and $\varphi^\vee=\otimes_{v\in\tV_F^\fin}\varphi^\vee_v\in\cV^\vee$ with $\varphi_v\in\pi_v$ and $\varphi^\vee_v\in\pi^\vee_v$ satisfying
\begin{itemize}
  \item[(T1)] for $v\in\tP$, both $\varphi_v$ and $\varphi^\vee_v$ are (nonzero) ordinary vectors;

  \item[(T2)] for $v\in\tV_F^\spl\setminus\tP$, the pair $(\varphi_v,\varphi_v^\vee)$ satisfies the conclusion of Proposition \ref{pr:matrix};

  \item[(T3)] for $v\in\tV_F^\fin\setminus\tV_F^\spl$, $\alpha(\varphi_v,\varphi_v^\vee)\neq 0$.
\end{itemize}
By Proposition \ref{pr:matrix} and Lemma \ref{le:matrix2}, such choice is possible. Put
\begin{align}\label{eq:measure}
\sL^0_{\cE_\tP}(\Pi)\coloneqq
\prod_{v\in\tP}\frac{\gamma(\Pi_v)}{\langle\varphi^\vee_v,\varphi_v\rangle_{\pi_v}}
\cdot\prod_{v\in\tV_F^\fin\setminus\tV_F^\spl}\frac{\alpha(\varphi_v,\varphi_v^\vee)^{-1}\cdot\Delta_{n+1,v}\cdot L(\tfrac{1}{2},\Pi_{n,v}\times\Pi_{n+1,v})}
{L(1,\Pi_{n,v},\As^{(-1)^n})L(1,\Pi_{n+1,v},\As^{(-1)^{n+1}})}\cdot\sP_\varphi\cdot(\sP_{\varphi^\vee})^\dag,
\end{align}
where $\gamma(\Pi_v)\in\dL^\times$ is the constant in Proposition \ref{pr:ordinary} and $\dag$ is the involution introduced in Notation \ref{no:dagger}. Note that by Lemma \ref{le:matrix2}, Lemma \ref{le:ordinary2}, and (T3), the above expression makes sense.

To show the interpolation property, we suppress $\iota$ from the notation. In particular, all representations have coefficients in $\dC$. By \eqref{eq:measure}, for $\chi$ as in the theorem, we have
\begin{align*}
\sL^0_{\cE_\tP}(\Pi)(\chi)
&=\prod_{v\in\tP}\frac{\gamma(\Pi_v)}{\langle\varphi^\vee_v,\varphi_v\rangle_{\pi_v}}
\(\frac{q_v^{\frac{n(n+1)(2n+1)}{3}}}{\omega(\Pi_v)}\)^{\fc_v}
\cdot\prod_{v\in\tV_F^\fin\setminus\tV_F^\spl}\frac{\alpha(\varphi_v,\varphi_v^\vee)^{-1}\cdot\Delta_{n+1,v}\cdot L(\tfrac{1}{2},\Pi_{n,v}\times\Pi_{n+1,v})}{L(1,\Pi_{n,v},\As^{(-1)^n})L(1,\Pi_{n+1,v},\As^{(-1)^{n+1}})} \\
&\times\int_{H(F)\backslash H(\dA_F^\infty)}\varphi\(h^\infty\cdot(1_n,\xi)\cdot[\varpi^\fc]\)\chi(\det h^\infty)\rd h^\infty \\
&\times\int_{H(F)\backslash H(\dA_F^\infty)}\varphi^\vee\(h^\infty\cdot(1_n,\xi)\cdot[\varpi^\fc]\)\chi(\det h^\infty)^{-1}\rd h^\infty.
\end{align*}

Now we apply the refined Gan--Gross--Prasad conjecture (that is, the Ichino--Ikeda conjecture), which has been fully proved for $G$ in \cites{BPLZZ,BPCZ}. Let $\tS$ be a subset of $\tV_F^\fin$ containing $\tP$ such that for $v\in\tV_F^\fin\setminus\tS$, we have
\[
\alpha^{\chi_v}(\varphi_v,\varphi_v^\vee)=
\frac{\Delta_{n+1,v}\cdot L(\tfrac{1}{2},(\Pi_{n,v}\otimes\wt{\chi_v})\times\Pi_{n+1,v})}{L(1,\Pi_{n,v},\As^{(-1)^n})L(1,\Pi_{n+1,v},\As^{(-1)^{n+1}})}.
\]
By the refined formula, we have
\begin{align*}
&\quad\int_{H(F)\backslash H(\dA_F^\infty)}\varphi\(h^\infty\cdot(1_n,\xi)\cdot[\varpi^\fc]\)\chi(\det h^\infty)\rd h^\infty \\
&\times\int_{H(F)\backslash H(\dA_F^\infty)}\varphi^\vee\(h^\infty\cdot(1_n,\xi)\cdot[\varpi^\fc]\)\chi(\det h^\infty)^{-1}\rd h^\infty \\
&=\frac{1}{2^{d(\Pi_n)+d(\Pi_{n+1})}}\frac{\Delta_{n+1}^\tS\cdot L(\tfrac{1}{2},(\Pi_n\otimes\wt\chi)^\tS\times\Pi_{n+1}^\tS)}{L(1,\Pi_n^\tS,\As^{(-1)^n})L(1,\Pi_{n+1}^\tS,\As^{(-1)^{n+1}})}
\prod_{v\in\tS}\alpha^{\chi_v}(\varphi_v,\varphi_v^\vee).
\end{align*}
Plugging in and by (T2), we have
\begin{align*}
&\sL^0_{\cE_\tP}(\Pi)(\chi) \\
&=\frac{1}{2^{d(\Pi_n)+d(\Pi_{n+1})}}\frac{\Delta_{n+1}^\tP\cdot L(\tfrac{1}{2},(\Pi_n\otimes\wt\chi)^\tP\times\Pi_{n+1}^\tP)}{L(1,\Pi_n^\tP,\As^{(-1)^n})L(1,\Pi_{n+1}^\tP,\As^{(-1)^{n+1}})}
\prod_{v\in\tP}\frac{\gamma(\Pi_v)}{\langle\varphi^\vee_v,\varphi_v\rangle_{\pi_v}}
\(\frac{q_v^{\frac{n(n+1)(2n+1)}{3}}}{\omega(\Pi_v)}\)^{\fc_v}\alpha^{\chi_v}(\varphi_v,\varphi_v^\vee).
\end{align*}
Finally, by (T1) and Proposition \ref{pr:ordinary}, we have
\begin{align*}
\sL^0_{\cE_\tP}(\Pi)(\chi)
=\frac{1}{2^{d(\Pi_n)+d(\Pi_{n+1})}}\frac{\Delta_{n+1}\cdot L(\tfrac{1}{2},(\Pi_n\otimes\wt\chi)^\tP\times\Pi_{n+1}^\tP)}{L(1,\Pi_n,\As^{(-1)^n})L(1,\Pi_{n+1},\As^{(-1)^{n+1}})}
\prod_{v\in\tP}\(\frac{q_v^{\frac{n(n+1)(2n+1)}{6}}}{\omega(\Pi_v)}\)^{\fc_v}.
\end{align*}
The interpolation formula follows as $L(\tfrac{1}{2},(\Pi_{n,v}\otimes\wt{\chi_v})\times\Pi_{n+1,v})=1$ (since $\chi_v$ is ramified) for every $v\in\tP$.
\end{proof}

We have the following corollary on the functional equation satisfied by $\sL^0_{\cE_\tP}(\Pi)$.

\begin{corollary}
We have
\[
\sL^0_{\cE_\tP}(\Pi)^\dag=\sL^0_{\cE_\tP}(\Pi^\vee).
\]
\end{corollary}

\begin{proof}
It suffices to show that for every finite character $\chi\colon\Gamma_\tP\to\dL_\chi^\times$ of conductor $\prod\fp_v^{\fc_v}$ for a tuple $\fc=(\fc_v)_{v\in\tP}$ of positive integers indexed by $\tP$ and every embedding $\iota\colon\dL_\chi\to\dC$, we have $\iota\sL^0_{\cE_\tP}(\Pi)(\chi)=\iota\sL^0_{\cE_\tP}(\Pi^\vee)(\chi^{-1})$. By Theorem \ref{th:function}, this follows from that $\omega(\Pi_v)=\omega(\Pi^\vee_v)$ for $v\in\tP$, that $d(\Pi_N)=d(\Pi^\vee_N)$ for $N=n,n+1$, that $L(1,\iota\Pi_N,\As^{(-1)^N})=L(1,\iota\Pi_N^\vee,\As^{(-1)^N})$ for $N=n,n+1$, and the classical functional equation
\[
L(\tfrac{1}{2},\iota(\Pi_n\otimes\wt\chi)\times\iota\Pi_{n+1})
=L(\tfrac{1}{2},\iota(\Pi_n^\vee\otimes\wt\chi^{-1})\times\iota\Pi^\vee_{n+1}).
\]
\end{proof}

\begin{notation}\label{no:quotient}
For every $\tP$-extension $\cE/E$ (Definition \ref{de:extension}), we denote by $\sL^0_\cE(\Pi)$ the image of $\sL^0_{\cE_\tP}(\Pi)$ under the natural homomorphism $\dL[[\Gamma_\tP]]^\circ\to\dL[[\Gal(\cE/E)]]^\circ$.
\end{notation}

\begin{remark}
We warn the readers that for a subset $\tP'$ of $\tP$, $\sL^0_{\cE_{\tP'}}(\Pi)$ is generally different from the image of $\sL^0_{\cE_\tP}(\Pi)$ under the natural homomorphism $\dL[[\Gamma_\tP]]^\circ\to\dL[[\Gamma_{\tP'}]]^\circ$.
\end{remark}

We propose the following nonvanishing conjecture.

\begin{conjecture}\label{co:vanishing1}
Suppose that $\dV_\Pi\neq 0$ and condition ($*$) in Notation \ref{no:galois} holds. Then $\sL^0_\cE(\Pi)\neq 0$ for every $\tP$-extension $\cE/E$ that contains $\cE_{\tP'}$ for some nonempty subset $\tP'$ of $\tP$.
\end{conjecture}

The conjecture is known for $n=1$ \cite{CV07}.

\begin{remark}\label{re:vanishing1}
Take an arbitrary embedding $\iota\colon\dL\to\dC$. Then $\epsilon(\frac{1}{2},\Pi_n^{(\iota)}\times\Pi_{n+1}^{(\iota)})$ equals $\epsilon(\Pi)$, which we have assumed to be $1$. Now $\epsilon(\frac{1}{2},\Pi_n^{(\iota)}\times\Pi_{n+1}^{(\iota)})$ decomposes as the product of $d(\Pi_n)d(\Pi_{n+1})$ root numbers (valued in $\{\pm 1\}$) for their isobaric factors. Then by the discussion in \cite{GGP12}*{\S26}, $\dV_\Pi\neq 0$ if and only if all those root numbers are equal to $1$.
\end{remark}

The proposition below will be used in \cite{LTX}, which can be skipped for now. Denote by $\tV_F^\Pi$ the finite subset of $\tV_F^\fin$ containing $\tV_F^\ram$ away from which $\Pi$ is unramified. For every $v\in\tV_F^\fin\setminus(\tV_F^\Pi\cup\tP)$, choose an integrally self-dual lattice $L_{n,v}$ of $V_{n,v}$ and put $K_v\coloneqq K_{n,v}\times K_{n+1,v}$, where $K_{n,v}$ and $K_{n+1,v}$ are the stabilizers of $L_{n,v}$ and $L_{n,v}\oplus O_{E_v}\cdot\te$, respectively.

\begin{proposition}
Suppose further that $\tP\cap\tV_F^\Pi=\emptyset$. There exists an element $\varphi=\otimes_{v\in\tV_F^\fin}\varphi_v\in\cV$ with $\varphi_v\in\pi_v$ an ordinary vector for $v\in\tP$ and $\varphi_v\in\pi_v^{K_v}$ for $v\in\tV_F^\fin\setminus(\tV_F^\Pi\cup\tP)$, such that the ideals of $\dL[[\Gamma_\tP]]^\circ$ generated by $\sL^0_{\cE_\tP}(\Pi)$ and $\sP_\varphi^2$ are the same.
\end{proposition}

\begin{proof}
We fix an MVW involution $*$ on $G$ \cite{MVW} that stabilizes $H$. For every $v\in\tV_F^\spl$, we
\begin{itemize}
  \item fix a standard isomorphism $G_v\simeq\GL_{n,F_v}\times\GL_{n+1,F_v}$ in the sense of Definition \ref{de:standard}, under which $*$ coincides with the transpose-inverse;

  \item denote by $\gamma_v$ the image of any uniformizer of $F_v$ in $\Gamma_\tP$ with respect to the isomorphism $F_v^\times\simeq E^{\times-}_v$ induced by the isomorphism above;

  \item fix an additive character $\psi_v\colon F_v\to\dC^\times$ of conductor $O_{F_v}$;

  \item denote by $\fc(\pi_v)\geq 0$ the Rankin--Selberg conductor of $\pi_v$, that is, the one satisfying that for every embedding $\iota\colon\dL\to\dC$, $q_v^{\fc(\pi_v)s}\cdot\epsilon(s,\iota\pi_{n,v}\times\iota\pi_{n+1,v},\psi_v)$ is a constant; it equals $0$ for $v\not\in\tV_F^\Pi$.
\end{itemize}
Put $\gamma\coloneqq\prod_{v\in\tV_F^\spl}\gamma_v^{\fc(\pi_v)}\in\Gamma_\tP$. We claim that there exists $\varphi$ as in the proposition such that
\begin{align}\label{eq:square}
\sP_\varphi^2=C\cdot\gamma\cdot\sL^0_{\cE_\tP}(\Pi)
\end{align}
for some constant $C\in\dL^\times$.

To check \eqref{eq:square}, we may fix an embedding $\dL\hookrightarrow\dC$ and show that for every finite characters $\chi\colon\Gamma_\tP\to\dC^\times$ that is \emph{ramified} everywhere in $\tP$,
\begin{align}\label{eq:square1}
\sP_\varphi(\chi)^2=C\cdot\chi(\gamma)\cdot\sL^0_{\cE_\tP}(\Pi)(\chi)
\end{align}
holds for some constant $C\in\dC^\times$. For every $v\in\tV_F^\spl$, we fix Whittaker models $\pi_v\otimes_\dL\dC\simeq\cW(\pi_v)_{\psi_v}$ and $\pi^\vee_v\otimes_\dL\dC\simeq\cW(\pi_v^\vee)_{\psi_v^{-1}}$, and will denote by $W_-$ the corresponding Whittaker function. Note that for every $W\in\cW(\pi_v)_{\psi_v}$, we have $\widetilde{W}\in\cW(\pi_v^\vee)_{\psi_v^{-1}}$, where $\widetilde{W}$ satisfies $\widetilde{W}(g)=W(w\pres{\rt}g^{-1})$ in which $w$ denotes the anti-diagonal Weyl element in $\GL_n(F_v)\times\GL_{n+1}(F_v)$.

For every element $\varphi=\otimes_v\varphi_v\in\cV$, we have the function $\varphi^*$ defined by the formula $\varphi^*(g)=\varphi(g^*)$. Then $\varphi^*=\otimes_v\varphi^*_v$ is again decomposable and belongs to $\cV^\vee$. Now we may choose an element $\varphi=\otimes_v\varphi_v$ as in the proposition satisfying
\begin{enumerate}
  \item for every $v\in\tV_F^\spl\setminus\tP$, there exists a constant $c_v\in\dC^\times$ (which equals $1$ if $v\not\in\tV_F^\Pi$) such that
      \[
      \int_{U_n(F_v)\backslash\GL_n(F_v)}W_{\varphi_v}(h)\cdot|\det h|_F^s\rd h=c_v\cdot L(\tfrac{1}{2}+s,\pi_{n,v}\times\pi_{n+1,v})
      \]
      holds as meromorphic functions in the variable $s\in\dC$, where $U_n$ denotes the upper-triangular unipotent subgroup of $\GL_n$ regarded as the diagonal subgroup of $\GL_n\times\GL_{n+1}$.

  \item for every $v\in\tV_F^\fin\setminus\tV_F^\spl$, $\alpha(\varphi_v,\varphi^*_v)\neq 0$.
\end{enumerate}
For every $v\in\tV_F^\spl$, there exists a constant $c'_v\in\dC^\times$ such that $W_{\varphi^*_v}=c'_v\widetilde{W_{\varphi_v}}$.

Now we compute $\sP_\varphi(\chi)^2$ for a finite character $\chi\colon\Gamma_\tP\to\dC^\times$ of conductor $\prod\fp_v^{\fc_v}$ for a tuple $\fc=(\fc_v)_{v\in\tP}$ of positive integers indexed by $\tP$. We have
\begin{align*}
\sP_\varphi(\chi)^2&=\prod_{v\in\tP}\(\frac{q_v^{\frac{n(n+1)(2n+1)}{6}}}{\omega(\pi_v)}\)^{2\fc_v}\cdot
\(\int_{H(F)\backslash H(\dA_F^\infty)}\varphi_\fc(h^\infty)\chi(\det h^\infty)\rd h^\infty\)^2 \\
&=\prod_{v\in\tP}\(\frac{q_v^{\frac{n(n+1)(2n+1)}{6}}}{\omega(\pi_v)}\)^{2\fc_v}\cdot
\int_{H(F)\backslash H(\dA_F^\infty)}\varphi_\fc(h^\infty)\chi(\det h^\infty)\rd h^\infty
\int_{H(F)\backslash H(\dA_F^\infty)}(\varphi_\fc)^*(h^\infty)\chi^{-1}(\det h^\infty)\rd h^\infty,
\end{align*}
where $\varphi_\fc\coloneqq\pi((1_n,\xi)\cdot[\varpi^\fc])\varphi$. Then by property (2) in the choice of $\varphi$, the refined Gan--Gross--Prasad formula, and the computation in Proposition \ref{pr:matrix}, there exists a constant $C\in\dC^\times$, independent of $\chi$, such that
\begin{align*}
\sP_\varphi(\chi)^2&=C\cdot\prod_{v\in\tP}\(\frac{q_v^{\frac{n(n+1)(2n+1)}{6}}}{\omega(\pi_v)}\)^{2\fc_v}\cdot
L(\tfrac{1}{2},(\Pi_n\otimes\wt\chi)\times\Pi_{n+1}) \\
&\times\prod_{v\in\tV_F^\spl}
\frac{\int_{U_n(F_v)\backslash\GL_n(F_v)}W_{\varphi_{\fc,v}}(h)\cdot\chi_v(\det h)\rd h}{L(\tfrac{1}{2},(\pi_{n,v}\otimes\chi_v)\times\pi_{n+1,v})}
\frac{\int_{U_n(F_v)\backslash\GL_n(F_v)}\widetilde{W_{\varphi_{\fc,v}}}(h)\cdot\chi_v^{-1}(\det h)\rd h}{L(\tfrac{1}{2},(\pi^\vee_{n,v}\otimes\chi^{-1}_v)\times\pi^\vee_{n+1,v})}.
\end{align*}
Now for $v\in\tV_F^\spl\setminus\tP$, the functional equation of the local Rankin--Selberg integrals implies that
\begin{align*}
&\quad\frac{\int_{U_n(F_v)\backslash\GL_n(F_v)}\widetilde{W_{\varphi_{\fc,v}}}(h)\cdot\chi_v^{-1}(\det h)\rd h}{L(\tfrac{1}{2},(\pi^\vee_{n,v}\otimes\chi^{-1}_v)\times\pi^\vee_{n+1,v})} \\
&=\varepsilon(\tfrac{1}{2},(\pi_{n,v}\otimes\chi_v)\times\pi_{n+1,v},\psi_v)
\cdot\frac{\int_{U_n(F_v)\backslash\GL_n(F_v)}W_{\varphi_{\fc,v}}(h)\cdot\chi_v(\det h)\rd h}{L(\tfrac{1}{2},(\pi_{n,v}\otimes\chi_v)\times\pi_{n+1,v})} \\
&=c''_v\cdot\chi_v(\gamma_v^{\fc(\pi_v)})\cdot\frac{\int_{U_n(F_v)\backslash\GL_n(F_v)}W_{\varphi_{\fc,v}}(h)\cdot\chi_v(\det h)\rd h}{L(\tfrac{1}{2},(\pi_{n,v}\otimes\chi_v)\times\pi_{n+1,v})}
\end{align*}
for a constant $c''_v\in\dC^\times$ independent of $\chi$. Then by property (1) in the choice of $\varphi$, one can modify the constant $C\in\dC^\times$, independent of $\chi$, such that
\begin{align*}
\sP_\varphi(\chi)^2&=C\cdot\chi(\gamma)\cdot\prod_{v\in\tP}\(\frac{q_v^{\frac{n(n+1)(2n+1)}{6}}}{\omega(\pi_v)}\)^{2\fc_v}\cdot
L(\tfrac{1}{2},(\Pi_n\otimes\wt\chi)\times\Pi_{n+1}) \\
&\times\prod_{v\in\tP}
\frac{\int_{U_n(F_v)\backslash\GL_n(F_v)}W_{\varphi_{\fc,v}}(h)\cdot\chi_v(\det h)\rd h}{L(\tfrac{1}{2},(\pi_{n,v}\otimes\chi_v)\times\pi_{n+1,v})}
\frac{\int_{U_n(F_v)\backslash\GL_n(F_v)}\widetilde{W_{\varphi_{\fc,v}}}(h)\cdot\chi_v^{-1}(\det h)\rd h}{L(\tfrac{1}{2},(\pi^\vee_{n,v}\otimes\chi^{-1}_v)\times\pi^\vee_{n+1,v})}.
\end{align*}
Now for $v\in\tP$, since $\pi_v$ is unramified, a direct computation shows that
\[
\varepsilon(\tfrac{1}{2},(\pi_{n,v}\otimes\chi_v)\times\pi_{n+1,v},\psi_v)=\(\frac{\omega(\pi_v)}{\omega(\pi_v^\vee)}\)^{\fc_v}
\cdot\(q_v^{-\frac{n(n+1)}{2}}\)^{\fc_v}\cdot G_{\psi_v^{-1}}(\chi_v^{-1})^{n(n+1)}.
\]
By \cite{Jan11}*{Corollary~2.8}, there exists a constant $c''_v\in\dC^\times$ such that
\[
\frac{\int_{U_n(F_v)\backslash\GL_n(F_v)}W_{\varphi_{\fc,v}}(h)\cdot\chi_v(\det h)\rd h}{L(\tfrac{1}{2},(\pi_{n,v}\otimes\chi_v)\times\pi_{n+1,v})}
=c''_v\cdot\(q_v^{-\frac{n(n+1)(n+2)}{6}}\)^{\fc_v}\cdot G_{\psi_v}(\chi_v)^{\frac{n(n+1)}{2}},
\]
hence
\[
\frac{\int_{U_n(F_v)\backslash\GL_n(F_v)}\widetilde{W_{\varphi_{\fc,v}}}(h)\cdot\chi_v^{-1}(\det h)\rd h}{L(\tfrac{1}{2},(\pi^\vee_{n,v}\otimes\chi^{-1}_v)\times\pi^\vee_{n+1,v})}
=c''_v\cdot\(\frac{\omega(\pi_v)}{\omega(\pi_v^\vee)}\)^{\fc_v}\cdot\(q_v^{-\frac{n(n+1)(n+2)}{6}}\)^{\fc_v}\cdot G_{\psi_v^{-1}}(\chi_v^{-1})^{\frac{n(n+1)}{2}}.
\]
In summary, we can again modify the constant $C\in\dC^\times$, independent of $\chi$, such that
\begin{align*}
\sP_\varphi(\chi)^2&=C\cdot\chi(\gamma)\cdot L(\tfrac{1}{2},(\Pi_n\otimes\wt\chi)\times\Pi_{n+1})\cdot
\prod_{v\in\tP}\(\frac{q_v^{\frac{n(n+1)(n-1)}{3}}}{\omega(\pi_v)\omega(\pi_v^\vee)}\)^{\fc_v}
\cdot G_{\psi_v}(\chi_v)^{\frac{n(n+1)}{2}}G_{\psi_v^{-1}}(\chi_v^{-1})^{\frac{n(n+1)}{2}} \\
&=C\cdot\chi(\gamma)\cdot\prod_{v\in\tP}\(\frac{q_v^{\frac{n(n+1)(2n+1)}{6}}}{\omega(\Pi_v)}\)^{\fc_v}
\cdot L(\tfrac{1}{2},(\Pi_n\otimes\wt\chi)\times\Pi_{n+1}).
\end{align*}
Thus, \eqref{eq:square1} holds and the proposition is proved.
\end{proof}

\section{Iwasawa Selmer groups and height pairing}
\label{ss:selmer}

In this section, $\tP$ is again a subset of $\tP(\Pi)$ \eqref{eq:ordinaryp}.

Take a subfield $\cE\subseteq \cE_\tP$ containing $E$. Let $\dW$ be a geometric Galois representation of $E$ with coefficients in $\dL$.\footnote{Recall that this means the representation is continuous, unramified outside a finite set of nonarchimedean places of $E$, and is de Rham at every $p$-adic places of $E$.} Take a $\Gal(\ol\dQ/E)$-stable $\dL^\circ$-lattice $\dW^\circ$ of $\dW$.

For every finite extension $E'/E$ contained in $\dC$ and every nonarchimedean place $u$ of $E'$, we have a distinguished subspace $\rH^1_f(E'_u,\dW)\subseteq\rH^1(E'_u,\dW)$ (known as the Bloch--Kato Selmer condition) as follows:
\[
\rH^1_f(E'_u,\dW)\coloneqq
\begin{dcases}
\rH^1_\unr(E'_u,\dW),& \text{if $u$ is not above $p$},\\
\Ker\(\rH^1(E'_u,\dW)\to\rH^1(E'_u,\dW\otimes_{\dQ_p}\dB_\cris)\),&\text{if $u$ is above $p$}.
\end{dcases}
\]
We then define $\rH^1_f(E'_u,\dW/\dW^\circ)$ (resp.\ $\rH^1_f(E'_u,\dW^\circ)$) to be the propagation of $\rH^1_f(E'_u,\dW)$, that is, the image (resp.\ preimage) of $\rH^1_f(E'_u,\dW)$ along the natural map $\rH^1(E'_u,\dW)\to\rH^1(E'_u,\dW/\dW^\circ)$ (resp.\ $\rH^1(E'_u,\dW^\circ)\to\rH^1(E'_u,\dW)$). The global Bloch--Kato Selmer groups are defined as
\begin{align*}
\rH^1_f(E',\dW/\dW^\circ)&\coloneqq\Ker\(\rH^1(E',\dW/\dW^\circ)\to
\prod_{u<\infty}\frac{\rH^1(E'_u,\dW/\dW^\circ)}{\rH^1_f(E'_u,\dW/\dW^\circ)}\), \\
\rH^1_f(E',\dW^\circ)&\coloneqq\Ker\(\rH^1(E',\dW^\circ)\to
\prod_{u<\infty}\frac{\rH^1(E'_u,\dW^\circ)}{\rH^1_f(E'_u,\dW^\circ)}\).
\end{align*}
Finally, we put
\begin{align*}
\sX(\cE,\dW^\circ)&\coloneqq\varprojlim_{E\subseteq E'\subseteq\cE}\Hom_{\dL^\circ}\(\rH^1_f(E',\dW/\dW^\circ),\dL/\dL^\circ\),\\
\sS(\cE,\dW^\circ)&\coloneqq\varprojlim_{E\subseteq E'\subseteq\cE}\rH^1_f(E',\dW^\circ)
\end{align*}
as $\dL^\circ[[\Gal(\cE/E)]]$-modules, where the transition maps are the (dual of) restriction maps and the corestriction maps, respectively.

\begin{definition}\label{no:selmer}
We define
\begin{align*}
\sX(\cE,\dW)&\coloneqq\sX(\cE,\dW^\circ)\otimes_{\dL^\circ}\dL,\\
\sS(\cE,\dW)&\coloneqq\sS(\cE,\dW)\otimes_{\dL^\circ}\dL,
\end{align*}
which are finitely generated $\dL[[\Gal(\cE/E)]]^\circ$-modules independent of the choice of $\dW^\circ$.
\end{definition}

We say that $\dW$ is \emph{pure} if it satisfies conditions (B,D) in \cite{Nek93}*{(2.1.2)}. It is clear that if $\dW$ is pure, then so is $\dW^\vee(1)$.

\begin{example}\label{ex:galois}
Let $\dW_\Pi$ be the $\dL[\Gal(\ol\dQ/E)]$-module underlying the representation $\rho_{\Pi_n}\otimes\rho_{\Pi_{n+1}}(n)$ (Notation \ref{no:galois}).
\begin{enumerate}
  \item First, $\dW_\Pi$ is pure. Indeed, by \cite{Car12}*{Theorem~1.1~\&~Theorem~1.2}, \cite{Car14}*{Theorem~1.1}, and \cite{TY07}*{Lemma~1.4}, we know that $\rho_{\Pi_n}\otimes\rho_{\Pi_{n+1}}(n)$ is pure of weight $-1$ (in the sense of \cite{TY07}) at every $u\in\tV_E^\fin$, which implies that $\dW_\Pi$ is pure.

  \item Second, we have for every finite extension $E'/E$ contained in $\cE_\tP$ a \emph{conjugation isomorphism} $\rH^1(E',\dW_\Pi)\xrightarrow\sim\rH^1(E',\dW_\Pi^\tc)\simeq\rH^1(E',\dW_{\Pi^\vee})$ by taking $\tc$-conjugation of extension classes; it is $\dL$-linear, but \emph{inverse} $\Gal(E'/E)$-equivariant.

  \item The system of conjugation isomorphisms induces an isomorphism
      \[
      -^\dag\colon\sS(\cE,\dW_{\Pi^\vee})\xrightarrow\sim\sS(\cE,\dW_\Pi)
      \]
      that is $(\dL[[\Gal(\cE/E)]]^\circ,\dag)$-linear.
\end{enumerate}

\end{example}

\begin{notation}\label{no:iwasawa}
Suppose that $\Gal(\cE/E)$ is torsion free so that $\dL[[\Gal(\cE/E)]]^\circ$ is integral. For a finitely generated $\dL[[\Gal(\cE/E)]]^\circ$-module $\sM$, we denote by $\sM^\tor$ its maximal $\dL[[\Gal(\cE/E)]]^\circ$-torsion submodule, and by
\[
\Char(\sM^\tor)\subseteq\dL[[\Gal(\cE/E)]]^\circ
\]
the characteristic ideal of $\sM^\tor$ (in particular, $\Char\(\sM\)=\Char(\sM^\tor)$ when $\sM$ is already torsion).
\end{notation}

\begin{conjecture}[Iwasawa's main conjecture in the coherent case]\label{co:main1}
Suppose that $\Pi$ is coherent and $\cE/E$ is a $\tP$-extension (Definition \ref{de:extension}).
\begin{enumerate}
  \item The following are equivalent:
     \begin{enumerate}
       \item $\sL^0_\cE(\Pi)$ is nonzero;

       \item the $\dL[[\Gal(\cE/E)]]^\circ$-module $\sX(\cE,\dW_\Pi)$ is of rank zero;

       \item the $\dL[[\Gal(\cE/E)]]^\circ$-module $\sS(\cE,\dW_{\Pi})$ is of rank zero.
     \end{enumerate}

  \item When the equivalent conditions in (1) hold, $\Char\(\sX(\cE,\dW_\Pi)\)$ is generated by $\sL^0_\cE(\Pi)$.
\end{enumerate}
\end{conjecture}

\begin{remark}
When $\tP=\emptyset$ so that $\cE=E$, part (1) of the above conjecture recovers the rank zero case of the Beilinson--Bloch--Kato conjecture; and part (2) is trivial.
\end{remark}

At the end of this section, we use Nekov\'{a}\v{r}'s $p$-adic height pairing to produce an $\dL[[\Gal(\cE/E)]]^\circ$-sesquilinear (that is, linear in the first variable and adjoint linear in the second variable) pairing
\begin{align}\label{eq:height}
\bh^\dW_\cE\colon\sS(\cE,\dW)\times\sS(\cE,\dW^\vee(1))
\to\Gamma_\tP^E\otimes_{\dZ_p}\dL[[\Gal(\cE/E)]]^\circ
\end{align}
when $\dW$ is pure.

The construction requires a choice of splittings of Hodge filtrations of $\dW$ at every place of $E$ above $\tP$ (see \cite{Nek93}*{Theorem~2.2} for the notion of such splitting). Since $\dW$ is geometric, we may choose a finite extension $E_\sharp/E$ contained in $\ol\dQ$ such that $\dW$ is semistable at every $p$-adic place of $E_\sharp$.\footnote{When $\dW=\dW_\Pi$ in Example \ref{ex:galois}, $\dW$ is already semistable at places above $\tP$ since $\Pi_v$ is semistable for every $v\in\tP$.} For every finite extension $E'/E$ contained in $\cE$, put $E'_\sharp\coloneqq E'\otimes_EE_\sharp$ and denote by $\tP'_\sharp$ the inverse image of $\tP$ with respect to the finite \'{e}tale extension $E'_\sharp/F$.\footnote{We regard $E'_\sharp$ as a finite product of finite extensions $E_\sharp^i/E_\sharp$ contained in $\ol\dQ$. Accordingly, in what follows, we understand $\Gal(\ol\dQ/E'_\sharp)$ as $\prod_i\Gal(\ol\dQ/E_\sharp^i)$, $\rH^1_f(E'_\sharp,-)$ as $\bigoplus_i\rH^1_f(E_\sharp^i,-)$, $\Gamma_\tP^{E'_\sharp}$ as $\bigoplus_i\Gamma_\tP^{E_\sharp^i}$, etc.} Clearly, $\dW$ as an $\dL[\Gal(\ol\dQ/E'_\sharp)]$-module remains pure, and is semistable at every $p$-adic places of $E'_\sharp$; and the chosen splittings of Hodge filtrations induce the ones at every place in $\tP'_\sharp$. We then have the $p$-adic height pairing
\[
\rh^\dW_{E'_\sharp}\colon\rH^1_f(E'_\sharp,\dW)\times\rH^1_f(E'_\sharp,\dW^\vee(1))
\to\Gamma_\tP^{E'_\sharp}\otimes_{\dZ_p}\dL
\]
constructed in \cite{Nek93}*{\S7}. Take $\Gal(\ol\dQ/E)$-stable $\dL^\circ$-lattices $\dW_1$ and $\dW_2$ of $\dW$ and $\dW^\vee(1)$, respectively. Then by the construction of the $p$-adic height pairing, there exists an integer $\delta=\delta(\dW_1,\dW_2)$, depending only on $\dW_1,\dW_2$, such that $\rh^\dW_{E'_\sharp}$ restricts to a pairing
\[
\rh_{E'_\sharp}\colon\rH^1_f(E'_\sharp,\dW_1)\times\rH^1_f(E'_\sharp,\dW_2)\to
\Gamma_\tP^{E'_\sharp}\otimes_{\dZ_p}p^\delta\dL^\circ
\]
for every $E'$. Define
\[
\rh_{E'}\colon\rH^1_f(E',\dW_1)\times\rH^1_f(E',\dW_2)\to\Gamma_\tP^E\otimes_{\dZ_p}p^\delta\dL^\circ
\]
to be the composition of the restriction map from $\rH^1_f(E',-)$ to $\rH^1_f(E'_\sharp,-)$, the pairing $\rh_{E'_\sharp}$, the norm map $\Nm_{E'_\sharp/E}$, and the multiplication by $[E_\sharp:E]^{-1}$ (after possibly changing $\delta$). In particular, $\rh_{E'}$ does not depend on $E_\sharp$. We define
\[
\bh_{E'}\colon\rH^1_f(E',\dW_1)\times\rH^1_f(E',\dW_2)\to\Gamma_\tP^E\otimes_{\dZ_p}p^\delta\dL^\circ[\Gal(E'/E)]
\]
to be the $\dL^\circ[\Gal(E'/E)]$-sesquilinearization of $\rh_{E'}$, namely,
\[
\bh_{E'}(x,y)=\sum_{\varsigma\in\Gal(E'/E)}\rh_{E'}(\varsigma x,y)[\varsigma^{-1}]
\]
for every $(x,y)\in\rH^1_f(E',\dW_1)\times\rH^1_f(E',\dW_2)$. By the lemma below, we obtain the desired pairing \eqref{eq:height} after inverting $p$.

\begin{lem}
The collection of pairings $(\bh_{E'})_{E'}$ is compatible under the corestriction maps on the source and the natural projection maps on the target, hence defines an $\dL^\circ[[\Gal(\cE/E)]]$-sesquilinear pairing
\[
\bh_\cE\coloneqq\varprojlim_{E\subseteq E'\subseteq \cE}\bh_{E'}\colon\sS(\cE,\dW_1)\times\sS(\cE,\dW_2)
\to\Gamma_\tP^E\otimes_{\dZ_p}p^\delta\dL^\circ[[\Gal(\cE/E)]].
\]
\end{lem}

\begin{proof}
It suffices to show that for $E\subseteq E'\subseteq E''\subseteq \cE$,
\begin{align}\label{eq:height1}
\sum_{\varsigma\in\Gal(E''/E')}\rh_{E''}(\varsigma x,y)=\rh_{E'}\(\Cor^{E''}_{E'} x,\Cor^{E''}_{E'} y\)
\end{align}
holds for every $(x,y)\in\rH^1_f(E'',\dW_1)\times\rH^1_f(E'',\dW_2)$, where $\Res^{E'}_{E''}$ and $\Cor^{E''}_{E'}$ denote the corresponding restriction and corestriction maps, respectively. Put $x'\coloneqq\sum_{\varsigma\in\Gal(E''/E')}\varsigma x$, which is simply the restriction of $\Cor^{E''}_{E'}x\in\rH^1_f(E',\dW_1)$ to $\rH^1_f(E'',\dW_1)$. By definition,
\begin{align*}
\sum_{\varsigma\in\Gal(E''/E')}\rh_{E''}(\varsigma x,y)=\rh_{E''}(x',y)
=\Nm_{E''_\sharp/E}\rh_{E''_\sharp}(x',y)=\Nm_{E'_\sharp/E}\(\Nm_{E''_\sharp/E'_\sharp}\rh_{E''_\sharp}(x',y)\).
\end{align*}
Thus, \eqref{eq:height1} follows as
\[
\Nm_{E''_\sharp/E'_\sharp}\rh_{E''_\sharp}(x',y)
=\Nm_{E''_\sharp/E'_\sharp}\rh_{E''_\sharp}\(\Res^{E'}_{E''}\Cor^{E''}_{E'} x,y\)
=\rh_{E'_\sharp}\(\Cor^{E''}_{E'} x,\Cor^{E''}_{E'} y\)
\]
by the projection formula with respect to a field extension satisfied by local height pairings at every finite place.

The lemma follows.
\end{proof}

\section{Incoherent anticyclotomic $p$-adic $L$-function}
\label{ss:incoherent}

In this section, $\tP$ is again a subset of $\tP(\Pi)$ \eqref{eq:ordinaryp}.

We study the case where $\Pi=\Pi_n\boxtimes\Pi_{n+1}$ is incoherent. Then there exists a hermitian space $V_n$ over $E$ that has signature $(n-1,1)$ at the real place of $F$ induced by the default embedding $E\subseteq\dC$ and $(n,0)$ at other ones, unique up to isomorphism, such that $V_{n,v}$ is the prescribed hermitian space from Section \ref{ss:ggp} for every $v\in\tV_F^\fin$. Put $V_{n+1}\coloneqq V_n\oplus E\cdot\te$. Put
\[
G\coloneqq\rU(V_n)\times\rU(V_{n+1}),\quad
\pi\coloneqq\otimes_{v\in\tV_F^\fin}\pi_v
\]
which is an irreducible admissible representation of $G(\dA_F^\infty)$ with coefficients in $\dL$.

We have a system of Shimura varieties $\{X_K\}_K$ associated with $\Res_{F/\dQ}G$ indexed by neat open compact subgroups $K$ of $G(\dA_F^\infty)$, which are quasi-projective smooth schemes over $E$ of dimension $2n-1$ (see for example \cite{LTXZZ}*{\S3.2} for more details). Put
\[
\rH^i_?(\ol{X},\dL(j))\coloneqq\varinjlim_{K}\rH^i_?(\ol{X}_K,\dL(j))
\]
for $i,j\in\dZ$ and $?\in\{\;,c\}$, where $\ol{X}_K\coloneqq X_K\otimes_E\ol\dQ$.

\begin{lem}\label{le:vanishing}
Let $K$ be a neat open compact subgroup of $G(\dA_F^\infty)$.
\begin{enumerate}
  \item The natural map $\rH^i_c(\ol{X}_K,\dL)[\pi^K]\to\rH^i(\ol{X}_K,\dL)[\pi^K]$ is an isomorphism for every $i\in\dZ$.

  \item The $\dL[K\backslash G(\dA_F^\infty)/K]$-module $\pi^K$ appears in $\rH^{2n-1}_c(\ol{X}_K,\dL)$ semisimply.

  \item If $K=K_\Box K^\Box$ in which $K^\Box$ is hyperspecial maximal, then $\rH^i_c(\ol{X}_K,\dL)[(\pi^\Box)^{K^\Box}]$ vanishes unless $i=2n-1$.
\end{enumerate}
\end{lem}

\begin{proof}
For this lemma, we may fix an embedding $\dL\to\dC$ and regard that $\pi$ is defined over $\dC$.

By possibly shrinking $K$, we may assume that $K$ is of the form $K_n\times K_{n+1}$ in which $K_N$ is a neat open compact subgroup of $G_N(\dA_F^\infty)$ for $N=n,n+1$. By the K\"{u}nneth formula, it suffices to show the following statements for $N=n,n+1$:
\begin{itemize}
  \item[(1')] The natural map $\rH^i_c(\ol{X}_{K_N},\dC)[\pi_N^{K_N}]\to\rH^i(\ol{X}_{K_N},\dC)[\pi_N^{K_N}]$ is an isomorphism for every $i\in\dZ$.

  \item[(2')] The $\dC[K_N\backslash G_N(\dA_F^\infty)/K_N]$-module $\pi_N^{K_N}$ appears in $\rH^{N-1}_c(\ol{X}_{K_N},\dC)$ semisimply.

  \item[(3')] If $K_N=K_{N\Box}K_N^\Box$ in which $K_N^\Box$ is hyperspecial maximal, then $\rH^i_c(\ol{X}_{K_N},\dC)[(\pi_N^\Box)^{K_N^\Box}]$ vanishes unless $i=N-1$.
\end{itemize}
Here $X_{K_N}$ denotes the corresponding Shimura variety for $G_N$ and $\ol{X}_{K_N}\coloneqq X_{K_N}\otimes_E\ol\dQ$. Denote by $X^*_{K_N}$ the unique (smooth) toroidal compactification of $X_{K_N}$ (which is just $X_{K_N}$ when $G_N$ is anisotropic) and put $Z_{K_N}\coloneqq X^*_{K_N}\setminus X_{K_N}$. We claim that $\rH^i(Z_{K_N},\dC)[(\pi_N^\Box)^{K_N^\Box}]$ vanishes for every $i\in\dZ$. Assuming the claim, (1') is immediate; (2') follows from the Matsushima formula; and (3') follows from the Matsushima formula and the fact that every automorphic representation of $G_N(\dA_F)$ that is nearly equivalent to $\pi_N$ is tempered everywhere by the endoscopic classification \cite{KMSW}.

For the claim itself, it has essentially been proved in \cite{LTXZZ}*{Lemma~6.1.11}. Indeed, let $\pi'_N$ be an irreducible admissible representation of $G_N(\dA_F^\infty)$ that appears in $\rH^i(Z_{K_N},\dC)$. Then $\pi'_N$ is the finite part of a representation of $G_N(\dA_F)$ that is an irreducible subquotient of the parabolic induction of a cuspidal automorphic representation of $L_N(\dA_F)$, where $L_N$ is the unique proper Levi subgroup of $G_N$ up to conjugation. Since $\pi_N$ is the finite part of a tempered cuspidal automorphic representation of $G_N(\dA_F)$, $\pi'_N$ is not nearly equivalent to $\pi_N$ hence cannot appear in $\rH^i(Z_{K_N},\dC)[(\pi_N^\Box)^{K_N^\Box}]$. In other words, $\rH^i(Z_{K_N},\dC)[(\pi_N^\Box)^{K_N^\Box}]$ vanishes.

The lemma is proved.
\end{proof}

Put
\[
\dV_\Pi\coloneqq\Hom_{G(\dA_F^\infty)}\(\pi^\vee,\rH^{2n-1}_c(\ol{X},\dL(n))\),\quad
\dV_{\Pi^\vee}\coloneqq\Hom_{G(\dA_F^\infty)}\(\pi,\rH^{2n-1}_c(\ol{X},\dL(n))\).
\]
By the Poincar\'{e} duality and Lemma \ref{le:vanishing}(1,2), the $\dL[\Gal(\ol\dQ/E)]$-modules $\dV_{\Pi^\vee}$ and $(\dV_\Pi)^\vee(1)$ are isomorphic. We fix a $\Gal(\ol\dQ/E)$-equivariant perfect pairing
\[
\llangle\;,\;\rrangle_\Pi\colon\dV_{\Pi^\vee}\times\dV_\Pi\to\dL(1).
\]
By definition, we have maps
\[
\pi\otimes_\dL\dV_{\Pi^\vee}\to\rH^{2n-1}_c(\ol{X},\dL(n)),\quad
\pi^\vee\otimes_\dL\dV_\Pi\to\rH^{2n-1}_c(\ol{X},\dL(n))
\]
of $\dL[G(\dA_F^\infty)]$-modules. Using the Poincar\'{e} duality pairing on $\ol{X}_K$ and $\llangle\;,\;\rrangle_\Pi$, we have induced maps
\[
-_\star\colon\pi^?\to\Hom_{\dL}\(\varprojlim_K\rH^{2n-1}(\ol{X}_K,\dL(n)),\dV_{\Pi^?}\)
\]
for $?\in\{\;,\vee\}$, so that the image of $(\pi^?)^K$ is contained in $\Hom_{\dL}\(\rH^{2n-1}(\ol{X}_K,\dL(n)),\dV_{\Pi^?}\)$.

Denote by $H\subseteq G$ the graph of the natural embedding $\rU(V_n)\hookrightarrow\rU(V_{n+1})$, and fix a rational Haar measure on $H(\dA_F^{\infty,\tP})$. For every $v\in\tP$, fix a standard isomorphism $G_v\simeq\GL_{n,F_v}\times\GL_{n+1,F_v}$ (Definition \ref{de:standard}) and recall the open compact subgroups $I_v^{(r)}$ of $G(F_v)$ for $r>0$ introduced after Definition \ref{de:standard}. Put $I_\tP^{(r)}\coloneqq\prod_{v\in\tP}I_v^{(r)}$.

For every element $\varphi=\otimes_{v\in\tV_F^\fin}\varphi_v\in\pi$ with $\varphi_v\in\pi_v$ that is a (nonzero) ordinary vector for $v\in\tP$, we choose
\begin{itemize}
  \item a neat open compact subgroup $K^\tP$ of $G(\dA_F^{\infty,\tP})$ that fixes $\varphi^\tP$,

  \item a Hecke operator $\tt\in\dL[K^\tP\backslash G(\dA_F^{\infty,\tP})/K^\tP]$ that annihilates $\rH^{2n}(\ol{X}_{K^\tP I_\tP^{(1)}},\dL^\circ(n))$ and such that its action on $\pi^{K^\tP}$ is given by the multiplication by a constant $\omega_\tt\in\dL^\times$, which is possible by Lemma \ref{le:vanishing}(3),

  \item a $\Gal(\ol\dQ/E)$-stable $\dL^\circ$-lattice $\dV_\Pi^\circ$ in $\dV_\Pi$ that contains the image of $\omega_\tt^{-1}\rH^{2n-1}(\ol{X}_{K^\tP I^{(1)}_\tP},\dL^\circ(n))$ under $\varphi_\star$.
\end{itemize}
Note that by Lemma \ref{le:ordinary1}, $\varphi$ is fixed by $I_\tP^{(1)}$.

For every tuple $\fc=(\fc_v)_{v\in\tP}$ of positive integers indexed by $\tP$, we denote by $E_\fc$ the subfield of $\cE_\tP$ such that $\Gal(E_\fc/E)=\Gamma_\tP/\fU_\fc$, so that the union of all $E_\fc$ is $\cE_\tP$. We define a certain element
\[
\kappa(\varphi)'\in\varprojlim_\fc\rH^1_g(E_\fc,\dV_\Pi^\circ)
\]
as follows:\footnote{Here, $\rH^1_g$ is defined similarly as $\rH^1_f$ but replacing $\dB_\cris$ by $\dB_{\dr}$ in the Selmer condition at $p$-adic places.} Put $K^\tP_H\coloneqq K^\tP\cap H(\dA_F^{\infty,\tP})$. Take a tuple $\fc=(\fc_v)_{v\in\tP}$ of positive integers indexed by $\tP$. Put
\begin{align*}
\varphi_\fc&\coloneqq\pi((1_n,\xi)\cdot[\varpi^\fc])\varphi,\\
K_\fc&\coloneqq((1_n,\xi)\cdot[\varpi^\fc])I^{(1)}_\tP((1_n,\xi)\cdot[\varpi^\fc])^{-1},\\
K_{\fc,H}&\coloneqq K_\fc\cap H(F_\tP).
\end{align*}
Then we have the map
\begin{align*}
(\varphi_\fc)_\star\colon\rH^{2n-1}(\ol{X}_{K^\tP K_\fc},\dL(n))\to\dV_\Pi.
\end{align*}
By definition, we have a special homomorphism
\[
Y_{K_H^\tP K_{\fc,H}}\to X_{K^\tP K_\fc}
\]
which is finite and unramified, where $Y_\bullet$ denotes the system of Shimura varieties associated with $H$. By Lemma \ref{le:ordinary3}(1), the structure morphism $Y_{K_H^\tP K_{\fc,H}}\to\Spec E$ factors through $\Spec E_\fc$, which gives rise to a decomposition
\[
Y_{K_H^\tP K_{\fc,H}}\otimes_EE_\fc=\coprod_{\varsigma\in\Gal(E_\fc/E)}\pres{\varsigma}Z_{K_H^\tP K_{\fc,H}}
\]
into disjoint union of open and closed subschemes indexed by $\Gal(E_\fc/E)$ that is compatible with the action of $\Gal(E_\fc/E)$, normalized in the way that $\pres{1}Z_{K_H^\tP K_{\fc,H}}$ is the fiber over the diagonal $\Spec E_\fc\to\Spec(E_\fc\otimes_EE_\fc)$. We have the induced cycle class
\[
[\pres{1}Z_{K_H^\tP K_{\fc,H}}]\in\CH^n(X_{K^\tP K_\fc}\otimes_EE_\fc).
\]
By the choice of $\tt$, $\tt[\pres{1}Z_{K_H^\tP K_{\fc,H}}]$ induces an element
\[
\alpha(\tt[\pres{1}Z_{K_H^\tP K_{\fc,H}}])\in\rH^1_g(E_\fc,\rH^{2n-1}(\ol{X}_{K^\tP K_\fc},\dL^\circ(n)))
\]
via the Abel--Jacobi map (see the proof of \cite{GS23}*{Proposition~9.29} for the reason of being in $\rH^1_g$ in a similar case).

Now put
\[
\kappa(\varphi)'_\fc\coloneqq \prod_{v\in\tP}\(\frac{1}{\omega(\pi_v)}\)^{\fc_v}\cdot\omega_\tt^{-1}\cdot
(\varphi_\fc)_\star\alpha(\tt[\pres{1}Z_{K_H^\tP K_{\fc,H}}])\in\rH^1_g(E_\fc,\dV_\Pi^\circ),
\]
which depends on the choice of $K^\tP$ but not on $\tt$.

\begin{proposition}\label{pr:cycle}
The collection $(\kappa(\varphi)'_\fc)_\fc$ is compatible under corestriction maps hence gives an element $\kappa(\varphi)'\in\varprojlim_\fc\rH^1_g(E_\fc,\dV_\Pi^\circ)$.
\end{proposition}

\begin{proof}
It suffices to show the compatibility under the corestriction map $\Cor^{E_{\fc'}}_{E_\fc}$ for $\fc'$ satisfying $\fc'_v=\fc_v+1$ for exactly one element $v\in\tP$ and $\fc'_v=\fc_v$ for others. For $?\in\{\fc,\fc'\}$, denote by $\rH^{2n}(X_{K^\tP K_{\fc'}}\otimes_EE_?,\dL^\circ(n))^0$ the kernel of the restriction map
\[
\rH^{2n}(X_{K^\tP K_{\fc'}}\otimes_EE_?,\dL^\circ(n))\to\rH^{2n}(\ol{X}_{K^\tP K_{\fc'}},\dL^\circ(n)).
\]
Then we have the commutative diagram
\[
\xymatrix{
\rH^{2n}(X_{K^\tP K_{\fc'}}\otimes_EE_{\fc'},\dL^\circ(n))^0 \ar[d]_-{\Tr^{E_{\fc'}}_{E_\fc}}\ar[r]^-{\alpha} & \rH^1(E_{\fc'},\rH^{2n-1}(\ol{X}_{K^\tP K_{\fc'}},\dL^\circ(n))) \ar[d]^-{\Cor^{E_{\fc'}}_{E_\fc}} \\
\rH^{2n}(X_{K^\tP K_{\fc'}}\otimes_EE_\fc,\dL^\circ(n))^0 \ar[r]^-{\alpha} & \rH^1(E_\fc,\rH^{2n-1}(\ol{X}_{K^\tP K_{\fc'}},\dL^\circ(n)))
}
\]
in which $\Tr^{E_{\fc'}}_{E_\fc}$ denotes the trace map along the extension $E_{\fc'}/E_\fc$. Thus, it suffices to show that
\begin{align}\label{eq:cycle}
(\varphi_{\fc'})_\star\alpha(\tt[\coprod_{\varsigma\in\Gal(E_{\fc'}/E_\fc)}\pres{\varsigma}Z_{K_H^\tP K_{\fc',H}}])
=\omega(\pi_v)\cdot(\varphi_\fc)_\star\alpha(\tt[\pres{1}Z_{K_H^\tP K_{\fc,H}}]),
\end{align}
in which $\coprod_{\varsigma\in\Gal(E_{\fc'}/E_\fc)}\pres{\varsigma}Z_{K_H^\tP K_{\fc',H}}$ is regarded as an open and closed subscheme of $Y_{K_H^\tP K_{\fc',H}}\otimes_EE_\fc$. By Definition \ref{de:ordinary1}, we have
\begin{align}\label{eq:cycle1}
\omega(\pi_v)\varphi_\fc
=\sum_{u\in U(O_{F_v})/(U(O_{F_v})\cap [\varpi_v] U(O_{F_v})[\varpi_v]^{-1})}\pi_v((1_n,\xi)\cdot[\varpi^\fc]\cdot u\cdot[\varpi_v])\varphi.
\end{align}
By \cite{Jan15}*{Lemma~6.5}, we have for every $u\in U(O_{F_v})/(U(O_{F_v})\cap [\varpi_v] U(O_{F_v})[\varpi_v]^{-1})$,
\[
(1_n,\xi)\cdot[\varpi^\fc]\cdot u\cdot[\varpi_v]\in K_{\fc',H}\cdot(1_n,\xi)\cdot[\varpi^{\fc'}]\cdot I^{(1)}_v.
\]
By Lemma \ref{le:ordinary3}(2), we have the following commutative diagram
\[
\xymatrix{
Y_{K_H^\tP K_{\fc,H}} \ar[d] & Y_{K_H^\tP K_{\fc',H}} \ar[l]_-{\kappa}\ar[r]\ar[d] & Y_{K_H^\tP K_{\fc',H}} \ar[d] \\
X_{K^\tP K_\fc} & X_{K^\tP(K_\fc\cap K_{\fc'})} \ar[l]\ar[r] & X_{K^\tP K_\fc}
}
\]
in which $\kappa$ is finite of degree $q_v^{\frac{n(n+1)(2n+1)}{6}}$, which is same as the cardinality of
\[
U(O_{F_v})/(U(O_{F_v})\cap [\varpi_v] U(O_{F_v})[\varpi_v]^{-1}).
\]
Since $I^{(1)}_v$ fixes $\varphi$ and translations by elements in $K_{\fc',H}$ fix $\kappa^{-1}\pres{1}Z_{K_H^\tP K_{\fc,H}}$, \eqref{eq:cycle1} implies that
\begin{align*}
\omega(\pi_v)\cdot(\varphi_\fc)_\star\alpha(\tt[\pres{1}Z_{K_H^\tP K_{\fc,H}}])
=(\varphi_{\fc'})_\star\alpha(\tt[\kappa^{-1}\pres{1}Z_{K_H^\tP K_{\fc,H}}]).
\end{align*}
Finally, since
\[
\kappa^{-1}\pres{1}Z_{K_H^\tP K_{\fc,H}}=\coprod_{\varsigma\in\Gal(E_{\fc'}/E_\fc)}\pres{\varsigma}Z_{K_H^\tP K_{\fc',H}},
\]
\eqref{eq:cycle} follow. The proposition is proved.
\end{proof}

\begin{remark}
The cohomology classes $(\kappa(\varphi)'_\fc)_\fc$ and their norm compatibility relation have already been essentially worked out in \cite{Loe21}*{\S5.1.1}, at least when $F=\dQ$.
\end{remark}

In terms of the above proposition, we put
\[
\kappa(\varphi)\coloneqq\vol(K_H^\tP)\cdot\kappa(\varphi)'
\in\(\varprojlim_\fc\rH^1_g(E_\fc,\dV_\Pi^\circ)\)\otimes_{\dL^\circ}\dL,
\]
which depends only on $\varphi$.

\begin{hypothesis}\label{hy:galois}
Every irreducible $\dL[\Gal(\ol\dQ/E)]$-subquotient of $\dV_\Pi$ is isomorphic to a (canonical) direct summand of $\dW_\Pi$ (Example \ref{ex:galois}).
\end{hypothesis}

\begin{remark}
The above hypothesis is known when $n\leq 2$, or when $n>2$ and $F\neq\dQ$ by an ongoing work of Kisin--Shin--Zhu. In fact, there is a more precise version of the above hypothesis specifying such direct summand via Arthur's multiplicity formula; see \cite{LL}*{Hypothesis~6.6}.
\end{remark}

Under Hypothesis \ref{hy:galois}, the natural map
\[
\sS(\cE_\tP,\dV_\Pi)=\(\varprojlim_\fc\rH^1_f(E_\fc,\dV_\Pi^\circ)\)\otimes_{\dL^\circ}\dL\to
\(\varprojlim_\fc\rH^1_g(E_\fc,\dV_\Pi^\circ)\)\otimes_{\dL^\circ}\dL
\]
is an isomorphism. Indeed, it follows from the same statement for $\dW_\Pi$ as $\dW_\Pi$ is pure. Then for every $\tP$-extension $\cE/E$ (Definition \ref{de:extension}),  we denote by $\kappa_\cE(\varphi)$ the image of $\kappa(\varphi)$ under the natural map $\sS(\cE_\tP,\dV_\Pi)\to\sS(\cE,\dV_\Pi)$.

\begin{notation}\label{no:submodule}
Assume Hypothesis \ref{hy:galois} and further that $\dV_\Pi$ is a semisimple $\dL[\Gal(\ol\dQ/E)]$-module so that it is a direct submodule of $\dW_\Pi$. Let $\cE/E$ be a $\tP$-extension.
\begin{enumerate}
  \item Denote by $\sK(\cE,\dW_\Pi)$ the $\dL[[\Gal(\cE/E)]]^\circ$-submodule of $\sS(\cE,\dW_\Pi)$ generated by $\kappa_\cE(\varphi^\vee)^\dag$ (Example \ref{ex:galois}(3)) for all $\varphi^\vee=\otimes_{v\in\tV_F^\fin}\varphi^\vee_v\in\pi^\vee$ with $\varphi^\vee_v\in\pi^\vee_v$ that is an ordinary vector for $v\in\tP$ (with respect to all standard isomorphisms (Definition \ref{de:standard}) at places in $\tP$).

  \item Put $\sT(\cE,\dW_\Pi)\coloneqq\sS(\cE,\dW_\Pi)/\sK(\cE,\dW_\Pi)$.
\end{enumerate}
\end{notation}

The conjecture below can be regarded as a higher-dimensional analogue of Perrin-Riou's Heegner point main conjecture \cite{PR87}.

\begin{conjecture}[Iwasawa's main conjecture in the incoherent case]\label{co:main2}
Suppose that $\Pi$ is incoherent; that Hypothesis \ref{hy:galois} holds; and that $\dV_\Pi$ is a semisimple $\dL[\Gal(\ol\dQ/E)]$-module. Let $\cE/E$ be a $\tP$-extension.
\begin{enumerate}
  \item The following are equivalent:
     \begin{enumerate}
       \item $\sK(\cE,\dW_{\Pi})$ (Notation \ref{no:iwasawa}) is nonzero;

       \item the $\dL[[\Gal(\cE/E)]]^\circ$-module $\sX(\cE,\dW_\Pi)$ is of rank one;

       \item the $\dL[[\Gal(\cE/E)]]^\circ$-module $\sS(\cE,\dW_{\Pi})$ is of rank one.
     \end{enumerate}

  \item When the equivalent conditions in (1) hold, we have
      \[
      \Char\(\sX(\cE,\dW_\Pi)^\tor\)=\Char\(\sT(\cE,\dW_{\Pi})\)^2.
      \]
\end{enumerate}
\end{conjecture}

\begin{remark}
The original Perrin-Riou's conjecture (including the one for Shimura curves over $\dQ$, under some conditions) was proved in
\cites{BCK21,Wan21} (with one side of the divisibiilty already proved in \cites{How04,Fou13}).
\end{remark}

Let $\cE/E$ be a $\tP$-extension. In what follows, we define an analog of the $p$-adic $L$-function in the incoherent case, which we denote by $\sL^1_\cE(\Pi)$. Define $\sL^1_\cE(\Pi)=0$ when $\dV_\Pi=0$.

Now we assume $\dV_\Pi\neq 0$. Assume Hypothesis \ref{hy:galois}. We have a pairing
\[
\bh^{\dV_\Pi}_\cE\colon\sS(\cE,\dV_\Pi)\times\sS(\cE,\dV_\Pi)
\to\Gamma_\tP^E\otimes_{\dZ_p}\dL[[\Gal(\cE/E)]]^\circ
\]
from the discussion in the previous section, after we identify $\dV_{\Pi^\vee}$ with $\dV_\Pi^\vee(1)$ using $\llangle\;,\;\rrangle_\Pi$. Note that by Hypothesis \ref{hy:galois} and the fact that $\Pi_v$ is semi-stably ordinary at every $v\in\tP$, $\dV_\Pi$ satisfies the Panchishkin condition hence admits a canonical splitting of the Hodge filtration at every $v\in\tP$, which is the one we use to define the above pairing.

Choose a pair $\varphi=\otimes_{v\in\tV_F^\fin}\varphi_v\in\cV$ and $\varphi^\vee=\otimes_{v\in\tV_F^\fin}\varphi^\vee_v\in\cV^\vee$ with $\varphi_v\in\pi_v$ and $\varphi^\vee_v\in\pi^\vee_v$ satisfying
\begin{itemize}
  \item[(T1)] for $v\in\tP$, both $\varphi_v$ and $\varphi^\vee_v$ are (nonzero) ordinary vectors;

  \item[(T2)] for $v\in\tV_F^\spl\setminus\tP$, the pair $(\varphi_v,\varphi_v^\vee)$ satisfies the conclusion of Proposition \ref{pr:matrix};

  \item[(T3)] for $v\in\tV_F^\fin\setminus\tV_F^\spl$, $\alpha(\varphi_v,\varphi_v^\vee)\neq 0$.
\end{itemize}
By Proposition \ref{pr:matrix} and Lemma \ref{le:matrix2}, such choice is possible. Put
\begin{align*}
\sL^1_\cE(\Pi)&\coloneqq
\prod_{v\in\tP}\frac{\gamma(\Pi_v)}{\langle\varphi^\vee_v,\varphi_v\rangle_{\pi_v}}
\cdot\prod_{v\in\tV_F^\fin\setminus\tV_F^\spl}\frac{\alpha(\varphi_v,\varphi_v^\vee)^{-1}\cdot\Delta_{n+1,v}\cdot L(\tfrac{1}{2},\Pi_{n,v}\times\Pi_{n+1,v})}{L(1,\Pi_{n,v},\As^{(-1)^n})L(1,\Pi_{n+1,v},\As^{(-1)^{n+1}})} \\
&\times \Nm_{E/F}\bh^{\dV_\Pi}_\cE\(\kappa_\cE(\varphi),\kappa_\cE(\varphi^\vee)\)
\in\Gamma_\tP^F\otimes_{\dZ_p}\dL[[\Gal(\cE/E)]]^\circ,
\end{align*}
where $\gamma(\Pi_v)\in\dL^\times$ is the constant in Proposition \ref{pr:ordinary}. Note that by Lemma \ref{le:matrix2}, Lemma \ref{le:ordinary2}, and (T3), the above expression makes sense.

\begin{lem}
The element $\sL^1_\cE(\Pi)$ does not depend on the choice of $(\varphi,\varphi^\vee)$.
\end{lem}

\begin{proof}
It suffices to show that for every finite character $\chi\colon\Gal(\cE/E)\to\dL_\chi^\times$, the element
\[
\bh^{\dV_\Pi}_\cE\(\kappa_\cE(\varphi),\kappa_\cE(\varphi^\vee)\)(\chi)\in\Gamma_\tP^F\otimes_{\dZ_p}\dL_\chi
\]
is independent of $(\varphi,\varphi^\vee)$. Without lost of generality, we may assume $\langle\varphi^\vee_v,\varphi_v\rangle_{\pi_v}=1$ for every $v\in\tP$.

The assignment
\[
(\varphi^\tP,\varphi^{\vee\tP})\in\pi^\tP\times\pi^{\vee\tP}\mapsto
\bh^{\dV_\Pi}_\cE\(\kappa_\cE(\varphi),\kappa_\cE(\varphi^\vee)\)(\chi)
\]
defines an element in
\[
\Hom_{H(\dA_F^{\infty,\tP})\times H(\dA_F^{\infty,\tP})}\((\pi_\chi)^\tP\boxtimes(\pi_\chi)^{\vee\tP},\Gamma_\tP^F\otimes_{\dZ_p}\dL_\chi\).
\]
Thus, by Lemma \ref{le:matrix1}, there exists a constant $c\in\Gamma_\tP^F\otimes_{\dZ_p}\dL_\chi$, independent of the pair $(\varphi^\tP,\varphi^{\vee\tP})$, such that
\[
\bh^{\dV_\Pi}_\cE\(\kappa_\cE(\varphi),\kappa_\cE(\varphi^\vee)\)(\chi)=c\prod_{v\in\tV_F^\infty\setminus\tP}
\(\frac{\Delta_{n+1,v}L(\tfrac{1}{2},\Pi_{n,v}\times\Pi_{n+1,v})}{L(1,\Pi_{n,v},\As^{(-1)^n})L(1,\Pi_{n+1,v},\As^{(-1)^{n+1}})}\)^{-1}
\alpha(\varphi^\tP_v,\varphi^{\vee\tP}_v).
\]
The lemma then follows.
\end{proof}

\begin{remark}
It is easy to see that $\sL^1_\cE(\Pi)$ neither depends on the choice of the standard isomorphisms $G_v\simeq\GL_{n,F_v}\times\GL_{n+1,F_v}$ for $v\in\tP$. However, $\sL^1_\cE(\Pi)$ does depend on the choices of $\llangle\;,\;\rrangle_\Pi$ and a rational Haar measure on $H(\dA_F^{\infty,\tP})$.\footnote{Choosing $\llangle\;,\;\rrangle_\Pi$ is equivalent to choosing an $\dL$-valued Haar measure on $G(\dA_F^\infty)$.}
\end{remark}

The conjecture and the remark below are analogs of Conjecture \ref{co:vanishing1} and Remark \ref{re:vanishing1}, respectively.

\begin{conjecture}\label{co:vanishing2}
Assume Hypothesis \ref{hy:galois}. Suppose that $\dV_\Pi\neq 0$ and condition ($*$) in Notation \ref{no:galois} holds. Then $\sL^1_\cE(\Pi)\neq 0$ for every $\tP$-extension $\cE/E$ that contains $\cE_{\tP'}$ for some nonempty subset $\tP'$ of $\tP$.
\end{conjecture}

When $n=1$, the conjecture can be deduced from \cite{CV07}.

\begin{remark}
Take an arbitrary embedding $\iota\colon\dL\to\dC$. Then $\epsilon(\frac{1}{2},\Pi_n^{(\iota)}\times\Pi_{n+1}^{(\iota)})$ equals $\epsilon(\Pi)$, which we have assumed to be $-1$. Now $\epsilon(\frac{1}{2},\Pi_n^{(\iota)}\times\Pi_{n+1}^{(\iota)})$ decomposes as the product of $d(\Pi_n)d(\Pi_{n+1})$ root numbers (valued in $\{\pm 1\}$) for their isobaric factors. Then by a similar discussion in \cite{GGP12}*{\S26} using \cite{LL}*{Lemma~3.15}, $\dV_\Pi\neq 0$ if and only if those root numbers contain $-1$ exactly once.
\end{remark}

We propose the following conjecture, which is a variant of Conjecture \ref{co:main2} without using the conditionally defined module $\sK(\cE,\dW_\Pi)$.

\begin{conjecture}\label{co:main3}
Suppose that $\Pi$ is incoherent and that Hypothesis \ref{hy:galois} holds. Let $\cE/E$ be a $\tP$-extension.
\begin{enumerate}
  \item The following are equivalent:
     \begin{enumerate}
       \item $\sL^1_\cE(\Pi)$ is nonzero;

       \item the $\dL[[\Gal(\cE/E)]]^\circ$-module $\sX(\cE,\dW_\Pi)$ is of rank one;

       \item the $\dL[[\Gal(\cE/E)]]^\circ$-module $\sS(\cE,\dW_{\Pi})$ is of rank one.
     \end{enumerate}

  \item When the equivalent conditions in (1) hold, $\Char\(\sX(\cE,\dW_\Pi)^\tor\)$ is generated by the image of $\phi\sL^1_\cE(\Pi)$ in $\dL[[\Gal(\cE/E)]]^\circ$ for all $\dZ_p$-linear maps $\phi\colon\Gamma_\tP^F\to\dZ_p$.\footnote{If one assumes Leopoldt's conjecture, then the $\dZ_p$-rank of $\Gamma_\tP^F$ is one; hence it suffices to consider one such $\phi$ that is surjective.}
\end{enumerate}
\end{conjecture}

At last, we propose a conjecture relating $\sL^1_{\cE_\tP}(\Pi)$ to the derivative of the hypothetical (full) $\tP$-adic $L$-function of $\Pi$ along the anticyclotomic direction.

\begin{hypothesis}\label{hy:function}
In this hypothesis, we temporarily allow $\Pi$ to be either coherent or incoherent. There exists a unique measure $\sL^E_\tP(\Pi)$ on $\Gamma_\tP^E$ valued in $\dL$, such that for every finite character $\Xi\colon\Gamma_\tP^E\to\dL_\Xi^\times$ of conductor $\prod_{u\mid\tP}\fp_u^{\fc_u}$ for a tuple $\fc=(\fc_u)_{u\mid\tP}$ of positive integers indexed by places of $E$ above $\tP$ and every embedding $\iota\colon\dL_\Xi\to\dC$, we have
\[
\iota\sL^E_\tP(\Pi)(\Xi)=G(\iota\Xi)^{\frac{n(n+1)}{2}}
\prod_{u\mid\tP}\(\frac{q_u^{\frac{(n-1)n(n+1)}{6}}}{\iota\omega(\Pi_u)}\)^{\fc_u}
\frac{\Delta_{n+1}\cdot L(\tfrac{1}{2},\iota(\Pi_n\otimes\wt\chi)\times\iota\Pi_{n+1})}
{2^{d(\Pi_n)+d(\Pi_{n+1})}\cdot L(1,\iota\Pi_n,\As^{(-1)^n})L(1,\iota\Pi_{n+1},\As^{(-1)^{n+1}})},
\]
where $G(\iota\Xi)$ is the global Gauss sum defined on \cite{Jan16}*{Page~460} (which depends only on $\iota\Xi$).
\end{hypothesis}

\begin{remark}
When $d(\Pi_n)=d(\Pi_{n+1})=1$, we know
\begin{itemize}
  \item for every fixed embedding $\iota_0\colon\dL\to\dC$, the existence of the measure satisfying the interpolation property in the above hypothesis up to a constant in $\dC^\times$ for $\iota$ extending $\iota_0$ (this is due to a series of works (mainly) \cites{KMS00,Jan11,Jan15,Jan16});

  \item the existence of the measure after restriction to $\Gamma_\tP^F$ satisfying the interpolation property in the above hypothesis, under the condition that for every $v\in\tV_F^{(p)}\setminus\tP(\Pi)$, $v$ is nonsplit in $E$ and $\Pi_v$ is unramified \cite{DZ}*{Theorem~B}.
\end{itemize}
\end{remark}

Now when $\Pi$ is incoherent, $\sL^E_\tP(\Pi)$ vanishes along the homomorphism $\Nm_{E/F}^-\colon\Gamma_\tP^E\to\Gamma_\tP$. Since the conormal bundle of the rigid analytic spectrum of $\dL[[\Gamma_\tP]]^\circ$ in that of $\dL[[\Gamma_\tP^E]]^\circ$ is canonically the constant bundle $\Gamma_\tP^F\otimes_{\dZ_p}\dL$, we obtain an element
\[
\rd\sL^E_\tP(\Pi)\in\Gamma_\tP^F\otimes_{\dZ_p}\dL[[\Gamma_\tP]]^\circ.
\]

\begin{conjecture}\label{co:derivative}
Assume Hypothesis \ref{hy:galois} and Hypothesis \ref{hy:function}. Then
\[
\sL^1_{\cE_\tP}(\Pi)=C\cdot\rd\sL^E_\tP(\Pi),
\]
where $C$ is an explicit constant in $\dL^\times$ depending on the choices of $\llangle\;,\;\rrangle_\Pi$ and a rational Haar measure on $H(\dA_F^{\infty,\tP})$.
\end{conjecture}

Disegni proved an analog \cite{Dis17}*{Theorem~C.4} of the above conjecture when $n=1$ in terms of quaternionic Shimura curves; while the special case for modular curves under the classical Heegner condition was already known by Howard \cite{How05}. For general $n$, the above conjecture is known \cite{DZ}*{Theorem~D} after evaluating at everywhere unramified characters of $\Gamma_\tP$ under several conditions (for instance, $d(\Pi_n)=d(\Pi_{n+1})=1$, $E/F$ everywhere unramified, $\tP=\tP(\Pi)=\tV_F^{(p)}$, $\Pi_v$ is unramified for every $v\in\tV_F^{(p)}$, etc).

\appendix


\begin{bibdiv}
\begin{biblist}


\bib{Ber95}{article}{
   author={Bertolini, Massimo},
   title={Selmer groups and Heegner points in anticyclotomic $\bZ_p$-extensions},
   journal={Compositio Math.},
   volume={99},
   date={1995},
   number={2},
   pages={153--182},
   issn={0010-437X},
   review={\MR{1351834}},
}

\bib{BD96}{article}{
   author={Bertolini, M.},
   author={Darmon, H.},
   title={Heegner points on Mumford-Tate curves},
   journal={Invent. Math.},
   volume={126},
   date={1996},
   number={3},
   pages={413--456},
   issn={0020-9910},
   review={\MR{1419003}},
   doi={10.1007/s002220050105},
}





\bib{BD05}{article}{
   author={Bertolini, M.},
   author={Darmon, H.},
   title={Iwasawa's main conjecture for elliptic curves over anticyclotomic $\mathbb{Z}_p$-extensions},
   journal={Ann. of Math. (2)},
   volume={162},
   date={2005},
   number={1},
   pages={1--64},
   issn={0003-486X},
   review={\MR{2178960}},
   doi={10.4007/annals.2005.162.1},
}


\bib{BP15}{article}{
   author={Beuzart-Plessis, R.},
   title={Endoscopie et conjecture locale raffin\'{e}e de Gan-Gross-Prasad pour
   les groupes unitaires},
   language={French, with English summary},
   journal={Compos. Math.},
   volume={151},
   date={2015},
   number={7},
   pages={1309--1371},
   issn={0010-437X},
   review={\MR{3371496}},
   doi={10.1112/S0010437X14007891},
}


\bib{BP16}{article}{
   author={Beuzart-Plessis, Rapha\"{e}l},
   title={La conjecture locale de Gross-Prasad pour les repr\'{e}sentations
   temp\'{e}r\'{e}es des groupes unitaires},
   language={French, with English and French summaries},
   journal={M\'{e}m. Soc. Math. Fr. (N.S.)},
   date={2016},
   number={149},
   pages={vii+191},
   issn={0249-633X},
   isbn={978-2-85629-841-1},
   review={\MR{3676153}},
   doi={10.24033/msmf.457},
}



\bib{BPLZZ}{article}{
   author={Beuzart-Plessis, Rapha\"{e}l},
   author={Liu, Yifeng},
   author={Zhang, Wei},
   author={Zhu, Xinwen},
   title={Isolation of cuspidal spectrum, with application to the
   Gan-Gross-Prasad conjecture},
   journal={Ann. of Math. (2)},
   volume={194},
   date={2021},
   number={2},
   pages={519--584},
   issn={0003-486X},
   review={\MR{4298750}},
   doi={10.4007/annals.2021.194.2.5},
}



\bib{BPCZ}{article}{
   author={Beuzart-Plessis, Rapha\"{e}l},
   author={Chaudouard, Pierre-Henri},
   author={Zydor, Micha\l },
   title={The global Gan-Gross-Prasad conjecture for unitary groups: the
   endoscopic case},
   journal={Publ. Math. Inst. Hautes \'{E}tudes Sci.},
   volume={135},
   date={2022},
   pages={183--336},
   issn={0073-8301},
   review={\MR{4426741}},
   doi={10.1007/s10240-021-00129-1},
}


\bib{BCK21}{article}{
   author={Burungale, Ashay},
   author={Castella, Francesc},
   author={Kim, Chan-Ho},
   title={A proof of Perrin-Riou's Heegner point main conjecture},
   journal={Algebra Number Theory},
   volume={15},
   date={2021},
   number={7},
   pages={1627--1653},
   issn={1937-0652},
   review={\MR{4333660}},
   doi={10.2140/ant.2021.15.1627},
}


\bib{Car12}{article}{
   author={Caraiani, Ana},
   title={Local-global compatibility and the action of monodromy on nearby
   cycles},
   journal={Duke Math. J.},
   volume={161},
   date={2012},
   number={12},
   pages={2311--2413},
   issn={0012-7094},
   review={\MR{2972460}},
   doi={10.1215/00127094-1723706},
}


\bib{Car14}{article}{
   author={Caraiani, Ana},
   title={Monodromy and local-global compatibility for $l=p$},
   journal={Algebra Number Theory},
   volume={8},
   date={2014},
   number={7},
   pages={1597--1646},
   issn={1937-0652},
   review={\MR{3272276}},
   doi={10.2140/ant.2014.8.1597},
}



\bib{Cas}{article}{
   author={Casselman, W.},
   title={Introduction to the theory of admissible representations of $p$-adic groups},
   note={available at \url{https://personal.math.ubc.ca/~cass/research/pdf/p-adic-book.pdf}},
}



\bib{CV07}{article}{
   author={Cornut, Christophe},
   author={Vatsal, Vinayak},
   title={Nontriviality of Rankin-Selberg $L$-functions and CM points},
   conference={
      title={$L$-functions and Galois representations},
   },
   book={
      series={London Math. Soc. Lecture Note Ser.},
      volume={320},
      publisher={Cambridge Univ. Press, Cambridge},
   },
   isbn={978-0-521-69415-5},
   date={2007},
   pages={121--186},
   review={\MR{2392354}},
   doi={10.1017/CBO9780511721267.005},
}



\bib{Dis17}{article}{
   author={Disegni, Daniel},
   title={The $p$-adic Gross-Zagier formula on Shimura curves},
   journal={Compos. Math.},
   volume={153},
   date={2017},
   number={10},
   pages={1987--2074},
   issn={0010-437X},
   review={\MR{3692745}},
   doi={10.1112/S0010437X17007308},
}



\bib{DZ}{article}{
   author={Disegni, Daniel},
   author={Zhang, Wei},
   title={Gan--Gross--Prasad cycles and derivatives of $p$-adic $L$-functions},
   note={\href{https://arxiv.org/abs/2410.08401}{arXiv:2410.08401}},
}


\bib{Fou13}{article}{
   author={Fouquet, Olivier},
   title={Dihedral Iwasawa theory of nearly ordinary quaternionic
   automorphic forms},
   journal={Compos. Math.},
   volume={149},
   date={2013},
   number={3},
   pages={356--416},
   issn={0010-437X},
   review={\MR{3040744}},
   doi={10.1112/S0010437X12000619},
}




\bib{GGP12}{article}{
   author={Gan, Wee Teck},
   author={Gross, Benedict H.},
   author={Prasad, Dipendra},
   title={Symplectic local root numbers, central critical $L$ values, and
   restriction problems in the representation theory of classical groups},
   language={English, with English and French summaries},
   note={Sur les conjectures de Gross et Prasad. I},
   journal={Ast\'erisque},
   number={346},
   date={2012},
   pages={1--109},
   issn={0303-1179},
   isbn={978-2-85629-348-5},
   review={\MR{3202556}},
}


\bib{GS23}{article}{
   author={Graham, Andrew},
   author={Shah, Syed Waqar Ali},
   title={Anticyclotomic Euler systems for unitary groups},
   journal={Proc. Lond. Math. Soc. (3)},
   volume={127},
   date={2023},
   number={6},
   pages={1577--1680},
   issn={0024-6115},
   review={\MR{4673434}},
}

\bib{Har14}{article}{
   author={Harris, R. Neal},
   title={The refined Gross-Prasad conjecture for unitary groups},
   journal={Int. Math. Res. Not. IMRN},
   date={2014},
   number={2},
   pages={303--389},
   issn={1073-7928},
   review={\MR{3159075}},
   doi={10.1093/imrn/rns219},
}





\bib{Hid98}{article}{
   author={Hida, Haruzo},
   title={Automorphic induction and Leopoldt type conjectures for $\mathrm{GL}(n)$},
   note={Mikio Sato: a great Japanese mathematician of the twentieth
   century},
   journal={Asian J. Math.},
   volume={2},
   date={1998},
   number={4},
   pages={667--710},
   issn={1093-6106},
   review={\MR{1734126}},
   doi={10.4310/AJM.1998.v2.n4.a5},
}




\bib{How04}{article}{
   author={Howard, Benjamin},
   title={Iwasawa theory of Heegner points on abelian varieties of $\mathrm{GL}_2$ type},
   journal={Duke Math. J.},
   volume={124},
   date={2004},
   number={1},
   pages={1--45},
   issn={0012-7094},
   review={\MR{2072210}},
   doi={10.1215/S0012-7094-04-12411-X},
}


\bib{How05}{article}{
   author={Howard, Benjamin},
   title={The Iwasawa theoretic Gross-Zagier theorem},
   journal={Compos. Math.},
   volume={141},
   date={2005},
   number={4},
   pages={811--846},
   issn={0010-437X},
   review={\MR{2148200}},
   doi={10.1112/S0010437X0500134X},
}


\bib{JPSS83}{article}{
   author={Jacquet, H.},
   author={Piatetskii-Shapiro, I. I.},
   author={Shalika, J. A.},
   title={Rankin-Selberg convolutions},
   journal={Amer. J. Math.},
   volume={105},
   date={1983},
   number={2},
   pages={367--464},
   issn={0002-9327},
   review={\MR{701565}},
   doi={10.2307/2374264},
}



\bib{Jan11}{article}{
   author={Januszewski, Fabian},
   title={Modular symbols for reductive groups and $p$-adic Rankin-Selberg
   convolutions over number fields},
   journal={J. Reine Angew. Math.},
   volume={653},
   date={2011},
   pages={1--45},
   issn={0075-4102},
   review={\MR{2794624}},
   doi={10.1515/CRELLE.2011.018},
}


\bib{Jan15}{article}{
   author={Januszewski, Fabian},
   title={On $p$-adic $L$-functions for $\mathrm{GL}(n)\times\mathrm{GL}(n-1)$
   over totally real fields},
   journal={Int. Math. Res. Not. IMRN},
   date={2015},
   number={17},
   pages={7884--7949},
   issn={1073-7928},
   review={\MR{3404005}},
   doi={10.1093/imrn/rnu181},
}


\bib{Jan16}{article}{
   author={Januszewski, Fabian},
   title={$p$-adic $L$-functions for Rankin-Selberg convolutions over number
   fields},
   language={English, with English and French summaries},
   journal={Ann. Math. Qu\'{e}.},
   volume={40},
   date={2016},
   number={2},
   pages={453--489},
   issn={2195-4755},
   review={\MR{3529190}},
   doi={10.1007/s40316-016-0061-y},
}



\bib{KMSW}{article}{
   author={Kaletha, Tasho},
   author={Minguez, Alberto},
   author={Shin, Sug Woo},
   author={White, Paul-James},
   title={Endoscopic Classification of Representations: Inner Forms of Unitary Groups},
   note={\href{https://arxiv.org/abs/1409.3731}{arXiv:1409.3731}},
}



\bib{KMS00}{article}{
   author={Kazhdan, D.},
   author={Mazur, B.},
   author={Schmidt, C.-G.},
   title={Relative modular symbols and Rankin-Selberg convolutions},
   journal={J. Reine Angew. Math.},
   volume={519},
   date={2000},
   pages={97--141},
   issn={0075-4102},
   review={\MR{1739728}},
   doi={10.1515/crll.2000.019},
}






\bib{LL}{article}{
   author={Li, Chao},
   author={Liu, Yifeng},
   title={Chow groups and $L$-derivatives of automorphic motives for unitary groups},
   journal={Ann. of Math. (2)},
   volume={194},
   date={2021},
   number={3},
   pages={817--901},
   issn={0003-486X},
   doi={10.4007/annals.2021.194.3.6},
}


\bib{LS}{article}{
   author={Liu, Dongwen},
   author={Sun, Binyong},
   title={Nearly ordinary $p$-adic period integrals},
   note={in preparation},
}


\bib{LTX}{article}{
   author={Liu, Yifeng},
   author={Tian, Yichao},
   author={Xiao, Liang},
   title={Iwasawa's main conjecture for Rankin--Selberg motives in the anticyclotomic case},
   note={in preparation},
}




\bib{LTXZZ}{article}{
   author={Liu, Yifeng},
   author={Tian, Yichao},
   author={Xiao, Liang},
   author={Zhang, Wei},
   author={Zhu, Xinwen},
   title={On the Beilinson-Bloch-Kato conjecture for Rankin-Selberg motives},
   journal={Invent. Math.},
   volume={228},
   date={2022},
   number={1},
   pages={107--375},
   issn={0020-9910},
   review={\MR{4392458}},
   doi={10.1007/s00222-021-01088-4},
}




\bib{Loe21}{article}{
   author={Loeffler, David},
   title={Spherical varieties and norm relations in Iwasawa theory},
   language={English, with English and French summaries},
   journal={J. Th\'{e}or. Nombres Bordeaux},
   volume={33},
   date={2021},
   number={3},
   pages={1021--1043},
   issn={1246-7405},
   review={\MR{4402388}},
   doi={10.1007/s10884-020-09844-5},
}



\bib{MVW}{book}{
   author={M\oe glin, Colette},
   author={Vign\'eras, Marie-France},
   author={Waldspurger, Jean-Loup},
   title={Correspondances de Howe sur un corps $p$-adique},
   language={French},
   series={Lecture Notes in Mathematics},
   volume={1291},
   publisher={Springer-Verlag, Berlin},
   date={1987},
   pages={viii+163},
   isbn={3-540-18699-9},
   review={\MR{1041060}},
   doi={10.1007/BFb0082712},
}


\bib{Mok15}{article}{
   author={Mok, Chung Pang},
   title={Endoscopic classification of representations of quasi-split
   unitary groups},
   journal={Mem. Amer. Math. Soc.},
   volume={235},
   date={2015},
   number={1108},
   pages={vi+248},
   issn={0065-9266},
   isbn={978-1-4704-1041-4},
   isbn={978-1-4704-2226-4},
   review={\MR{3338302}},
   doi={10.1090/memo/1108},
}



\bib{Nek93}{article}{
   author={Nekov\'{a}\v{r}, Jan},
   title={On $p$-adic height pairings},
   conference={
      title={S\'{e}minaire de Th\'{e}orie des Nombres, Paris, 1990--91},
   },
   book={
      series={Progr. Math.},
      volume={108},
      publisher={Birkh\"{a}user Boston, Boston, MA},
   },
   date={1993},
   pages={127--202},
   review={\MR{1263527}},
   doi={10.1007/s10107-005-0696-y},
}




\bib{PR87}{article}{
   author={Perrin-Riou, Bernadette},
   title={Fonctions $L$ $p$-adiques, th\'{e}orie d'Iwasawa et points de Heegner},
   language={French, with English summary},
   journal={Bull. Soc. Math. France},
   volume={115},
   date={1987},
   number={4},
   pages={399--456},
   issn={0037-9484},
   review={\MR{928018}},
}


\bib{SV17}{article}{
   author={Sakellaridis, Yiannis},
   author={Venkatesh, Akshay},
   title={Periods and harmonic analysis on spherical varieties},
   language={English, with English and French summaries},
   journal={Ast\'{e}risque},
   number={396},
   date={2017},
   pages={viii+360},
   issn={0303-1179},
   isbn={978-2-85629-871-8},
   review={\MR{3764130}},
}




\bib{TY07}{article}{
   author={Taylor, Richard},
   author={Yoshida, Teruyoshi},
   title={Compatibility of local and global Langlands correspondences},
   journal={J. Amer. Math. Soc.},
   volume={20},
   date={2007},
   number={2},
   pages={467--493},
   issn={0894-0347},
   review={\MR{2276777}},
   doi={10.1090/S0894-0347-06-00542-X},
}



\bib{Wal03}{article}{
   author={Waldspurger, J.-L.},
   title={La formule de Plancherel pour les groupes $p$-adiques (d'apr\`es
   Harish-Chandra)},
   language={French, with French summary},
   journal={J. Inst. Math. Jussieu},
   volume={2},
   date={2003},
   number={2},
   pages={235--333},
   issn={1474-7480},
   review={\MR{1989693}},
   doi={10.1017/S1474748003000082},
}


\bib{Wan21}{article}{
   author={Wan, Xin},
   title={Heegner point Kolyvagin system and Iwasawa main conjecture},
   journal={Acta Math. Sin. (Engl. Ser.)},
   volume={37},
   date={2021},
   number={1},
   pages={104--120},
   issn={1439-8516},
   review={\MR{4204538}},
   doi={10.1007/s10114-021-8355-7},
}



\bib{Zha14}{article}{
   author={Zhang, Wei},
   title={Automorphic period and the central value of Rankin-Selberg
   L-function},
   journal={J. Amer. Math. Soc.},
   volume={27},
   date={2014},
   number={2},
   pages={541--612},
   issn={0894-0347},
   review={\MR{3164988}},
   doi={10.1090/S0894-0347-2014-00784-0},
}


\end{biblist}
\end{bibdiv}
\end{document}